\documentclass[11pt,reqno]{amsart}
\usepackage[nosumlimits]{mathtools}
\usepackage{amsthm}
\usepackage{float}
\usepackage{graphicx, enumerate, url}
\usepackage[margin=1in]{geometry}
\usepackage{amssymb,cool}
\usepackage{algorithm}
\usepackage{algpseudocode}
\usepackage{color}
\usepackage{tikz}
\usepackage{tikz-cd,mathtools}
\usepackage{bbm}
\usetikzlibrary{3d,calc}
\usepackage[low-sup]{subdepth}
\usepackage[normalem]{ulem}
\usepackage[font=small,skip=0pt]{caption}

\numberwithin{equation}{section}
\numberwithin{algorithm}{section}

\theoremstyle{plain}
\newtheorem{theorem}{Theorem}[section]
\newtheorem{proposition}[theorem]{Proposition}
\newtheorem{lemma}[theorem]{Lemma}
\newtheorem{corollary}[theorem]{Corollary}

\newtheorem{facts}[theorem]{Facts}

\theoremstyle{definition}
\newtheorem{definition}[theorem]{Definition}
\newtheorem{example}[theorem]{Example}

\theoremstyle{remark}

\newcommand{\tp}{{\scriptscriptstyle\mathsf{T}}}

\let\O\undefined
\let\E\undefined
\let\H\undefined

\DeclareMathOperator{\O}{O}
\DeclareMathOperator{\E}{E}
\DeclareMathOperator{\H}{H}

\DeclareMathOperator{\SO}{SO}
\DeclareMathOperator{\tr}{tr}
\DeclareMathOperator{\diag}{diag}
\DeclareMathOperator{\rank}{rank}

\DeclareMathOperator{\spn}{span}

\DeclareMathOperator{\GL}{GL}

\DeclareMathOperator{\V}{V}

\DeclareMathOperator{\Aut}{Aut}

\DeclareMathOperator{\Vol}{Vol}

\DeclareMathOperator{\im}{im}
\DeclareMathOperator{\Gr}{Gr}
\DeclareMathOperator{\Graff}{Graff}
\DeclareMathOperator{\Vaff}{Vaff}

\DeclareMathOperator{\Stab}{Stab}

\DeclareMathOperator{\GA}{GA}
\DeclareMathOperator{\St}{St}
\DeclareMathOperator{\Staff}{Staff}

\begin{document}
\title{The Grassmannian of affine subspaces}
\author[L.-H.~Lim]{Lek-Heng~Lim}
\address{Computational and Applied Mathematics Initiative, Department of Statistics,
University of Chicago, Chicago, IL 60637-1514.}
\email{lekheng@galton.uchicago.edu, kenwong@uchicago.edu}
\author[K.~S.-W.~Wong]{Ken~Sze-Wai~Wong}
\author[K.~Ye]{Ke Ye}
\address{KLMM, Academy of Mathematics and Systems Science, Chinese Academy of Sciences, Beijing 100190, China}
\email{keyk@amss.ac.cn}
\begin{abstract}
The Grassmannian of affine subspaces is a natural generalization of both the Euclidean space, points being $0$-dimensional affine subspaces, and the usual Grassmannian, linear subspaces being special cases of affine subspaces. We show that, like the Grassmannian, the affine Grassmannian has rich geometrical and topological properties: It has the structure of a homogeneous space, a differential manifold, an algebraic variety, a vector bundle, a classifying space, among many more structures; furthermore; it affords an analogue of Schubert calculus and its (co)homology and homotopy groups may be readily determined. On the other hand, like the Euclidean space, the affine Grassmannian serves as a concrete computational platform on which various distances, metrics, probability densities may be explicitly defined and computed via numerical linear algebra. Moreover, many standard problems in machine learning and statistics --- linear regression, errors-in-variables regression, principal components analysis, support vector machines, or more generally any problem that seeks linear relations among variables that either best represent them or separate them into components --- may be naturally formulated as problems on the affine Grassmannian.
\end{abstract}



\maketitle
\section{Introduction}

The Grassmannian of affine subspaces,  denoted $\Graff(k,n)$, is an analogue of the usual Grassmannian $\Gr(k, n)$. Just as $\Gr(k, n)$ parameterizes $k$-dimensional \emph{linear} subspaces in $\mathbb{R}^n$,  $\Graff(k,n)$ parameterizes $k$-dimensional \emph{affine} subspaces in $\mathbb{R}^n$, i.e., $\mathbb{A} + b$ where the $k$-dimensional linear subspace $\mathbb{A} \subseteq \mathbb{R}^n$ is translated by a displacement vector $b \in \mathbb{R}^n$. 

To the best of our knowledge, the Grassmannian of affine subspaces was  first described in an elegant little volume \cite{KR} based on  Gian-Carlo Rota's 1986 `Lezioni Lincee' lectures at the Scuola Normale Superiore. The treatment in \cite[pp.~86--87]{KR}  was somewhat cursory as $\Graff(k,n)$ played only an auxiliary role in Rota's lectures (on geometric probability). Aside from another equally brief mention  in \cite[Section~9.1.3]{Nicolaescu}, we are unaware of any other discussion. Compared to its universally known cousin $\Gr(k,n)$, it is fair to say that $\Graff(k,n)$ has received next to no attention.

The goal of our article is to fill this gap. We will show that the Grassmannian of affine subspaces has rich algebraic, geometric, and topological properties; moreover, it is an important object that could rival the usual Grassmannian in practical applicability, serving as a computational and modeling platform for problems in statistical estimation and pattern recognition.
We start by showing that $\Graff(k,n)$ may be viewed from several perspectives, and in more than a dozen ways:
\begin{description}
\item[algebra] as collections of (i) Minkowski sums of sets, (ii) cosets in an additive group, (iii) $n \times (k+1)$ matrices;

\item[differential geometry] as a  (iv) smooth manifold, (v) homogeneous space, (vi) Riemmannian manifold, (vii) base space of the compact and noncompact affine Stiefel manifolds regarded as principal bundles;

\item[algebraic geometry] as a (viii) irreducible nonsingular algebraic variety, (ix)  Zariski open  dense subset of the Grassmannian, (x) real affine variety of projection matrices;

\item[algebraic topology] as a (xi) vector bundle, (xii) classifying space.
\end{description}
$\Graff(k,n)$ may also be regarded, in an appropriate sense, as the complement of $\Gr(k+1,n)$ in $\Gr(k+1,n+1)$, or, in a different sense, as the moduli space of $k$-dimensional affine subspaces in $\mathbb{R}^n$.
Moreover one may readily define, calculate, and compute various objects on $\Graff(k,n)$  of either theoretical or practical interests:
\begin{description}
\item[Schubert calculus] affine (a) flags, (b) Schubert varieties, (c) Schubert cycles;

\item[algebraic topology] (d) homotopy, (e) homology, (f) cohomology groups/ring;

\item[metric geometry] (g) distances,  (h) geodesic, (i) metrics;

\item[probability] (j) uniform, (k) von Mises--Fisher, (l) Langevin--Gaussian distributions.
\end{description}
The main reason for our optimism that $\Graff(k,n)$ may be no less important than $\Gr(k,n)$ in applications is the observation that common problems in multivariate data analysis and machine learning are naturally optimization problems over $\Graff(k,n)$:
\begin{description}
\item[statistics] (1) linear regression, (2) error-in-variables regression, (3) principal component analysis, (4)  support vector machines.
\end{description}
In retrospect this is no surprise, many statistical estimation problems involve a search for linear relations among variables and are therefore ultimately a problem of finding one or more affine subspaces that either best represent a given data set (regression) or best separate it into two or more components (classification).

In a companion article \cite{WYL}, we showed that in practical terms,  optimization problems over $\Graff(k,n)$ are no different from optimization problems over $\mathbb{R}^n$, which is of course just $\Graff(0,n)$. More precisely, we showed that, like the Euclidean space $\mathbb{R}^n$, $\Graff(k,n)$ serves the role of a concrete computational platform on which tangent spaces, Riemannian metric, exponential maps, parallel transports, gradients and Hessians of real-valued functions, optimization algorithms such as steepest descent, conjugate gradient, Newton methods, may all be efficiently computed using only standard numerical linear algebra.

For brevity, we will use the term \emph{affine Grassmannian} when referring to the Grassmannian of affine subspaces from this point onwards.  The term is now used far more commonly to refer to another very different object \cite{AR, FG, L} but in this article, it will always be used in the sense of Definition~\ref{def:graff}. To resolve the conflicting nomenclature, an alternative might be to christen the Grassmannian of affine subspaces the \emph{Rota Grassmannian}.

Unless otherwise noted, the results in  this article have not appeared before elsewhere to the best of our knowledge, although some of them are certainly routine for the experts. We have written our article with the hope that it would also be read by applied and computational mathematicians, statisticians, and engineers ---  in an effort to improve its accessibility, we have provided more basic details than is customary.

\section{Basic terminologies}\label{sec:basic}

We remind the reader of some basic terminologies. A  \emph{$k$-plane}  is a $k$-dimensional linear subspace and  a \emph{$k$-flat} is a $k$-dimensional affine subspace.  A \emph{$k$-frame} is an ordered  basis of a $k$-plane and we will regard it as an $n \times k$ matrix whose columns $a_1,\dots, a_k$ are the basis vectors. A \emph{flag} is a strictly increasing sequence of nested linear subspaces, $\mathbb{A}_0\subseteq \mathbb{A}_{1} \subseteq \mathbb{A}_2\subseteq \cdots $. A flag is said to be \emph{complete} if $\dim \mathbb{A}_k= k$, \emph{finite} if $k =0,1,\dots,n$, and \emph{infinite} if $k\in \mathbb{N} \cup \{0\}$. 
Throughout this article, a blackboard bold letter $\mathbb{A}$ will always denote a subspace and the corresponding normal letter $A$ will then denote a matrix of basis vectors (often but not necessarily orthonormal) of $\mathbb{A}$.

We write $\Gr(k,n)$ for the \emph{Grassmannian} of $k$-planes in $\mathbb{R}^n$, $\V(k,n)$ for the \emph{Stiefel manifold} of orthonormal $k$-frames, and $\O(n) \coloneqq \V(n,n)$ for the \emph{orthogonal group}. We may regard $\V(k,n)$ as a homogeneous space,
\begin{equation}\label{eq:stief}
\V(k,n) \cong \O(n)/\O(n-k),
\end{equation}
or more concretely as the set of $n \times k$ matrices with orthonormal columns.  
There is a right action of the orthogonal group $\O(k)$ on $\V(k,n)$: For $Q\in \O(k)$ and $A \in \V(k,n)$, the action  yields $AQ \in \V(k,n)$ and the resulting homogeneous space is $\Gr(k,n)$, i.e.,
\begin{equation}\label{eq:grass}
\Gr(k,n) \cong \V(k,n)/\O(k) \cong \O(n)/\bigl(\O(n-k) \times \O(k)\bigr).
\end{equation}
So $\mathbb{A} \in \Gr(k,n)$ may be identified with the equivalence class of its orthonormal $k$-frames $\{ AQ \in \V(k,n): Q \in \O(k)\}$. Note that $\spn(AQ) = \spn(A)$ for $Q \in \O(k)$. 

There is also a purely algebraic counterpart to the last paragraph, useful for generalizing to $k$-planes in a vector space that may not have an inner product (e.g., over fields of nonzero characteristics). We follow the terminologies and notations in \cite[Section~2]{AMS}. The \emph{noncompact Stiefel manifold} of $k$-frames is $\St(k,n) $. It may regarded as a homogeneous space
\begin{equation}\label{eq:ncstief}
\St(k,n) \cong \GL(n)/\GL(n-k),
\end{equation}
or more concretely as the set of $n \times k$ matrices with full rank.  
There is a right action of the general linear group $\GL(k)$ on $\St(k,n)$: For $X\in \GL(k)$ and $A \in \St(k,n)$, the action  yields $AX \in \St(k,n)$ and the resulting homogeneous space is $\Gr(k,n)$, i.e.,
\begin{equation}\label{eq:ncgrass}
\Gr(k,n) \cong \St(k,n)/\GL(k) \cong \GL(n)/\bigl(\GL(n-k) \times \GL(k)\bigr).
\end{equation}
So $\mathbb{A} \in \Gr(k,n)$ may be identified with the equivalence class of its $k$-frames $\{ AX \in \St(k,n): X \in \GL(k)\}$. Note that $\spn(AX) = \spn(A)$ for $X \in \GL(k)$. The reader would see that orthogonality has been avoided in this paragraph.

\section{Algebra of the affine Grassmannian}\label{sec:alg}

We will begin by discussing the set-theoretic and algebraic properties of the affine Grassmannian and introducing its two infinite-dimensional counterparts.
\begin{definition}[Affine Grassmannian]\label{def:graff}
Let $k<n$ be positive integers. The \emph{Grassmannian of $k$-dimensional affine subspaces} in $\mathbb{R}^n$ or Grassmannian of $k$-flats in $\mathbb{R}^n$, denoted by $\Graff(k,n)$, is the set of all $k$-dimensional  affine subspaces of $\mathbb{R}^n$. For an abstract vector space $V$, we write $\Graff_k(V)$ for the set of $k$-flats in $V$.
\end{definition}
This set-theoretic definition hardly reveals anything about the rich algebra, geometry, and topology of the affine Grassmannian, which we will examine over this and the next few sections.

We denote a $k$-dimensional affine subspace as $\mathbb{A}+b \in \Graff(k,n)$ where $\mathbb{A} \in \Gr(k,n)$ is a $k$-dimensional linear subspace and $b \in \mathbb{R}^n$ is the displacement of $\mathbb{A}$ from the origin. If $A = [a_1,\dots, a_k] \in \mathbb{R}^{n \times k}$ is a basis of $\mathbb{A}$, then
\begin{equation}\label{eq:affine}
\mathbb{A}+b \coloneqq \{\lambda_1 a_1 + \dots + \lambda_k a_k + b \in \mathbb{R}^n : \lambda_1,\dots,\lambda_k \in \mathbb{R} \}.
\end{equation}
The notation $\mathbb{A} + b$ may be taken to mean (i) the \emph{Minkowski sum} of the sets $\mathbb{A}$ and $\{b\}$ in the Euclidean space $\mathbb{R}^n$, (ii) a \emph{coset} of the subgroup $\mathbb{A}$ in the additive group $\mathbb{R}^n$, or (iii) a coset of the subspace  $\mathbb{A}$ in the vector space $\mathbb{R}^n$. The dimension of  $\mathbb{A} + b$ is defined to be the dimension of the vector space $\mathbb{A}$.  As one would expect of a coset representative, the displacement vector $b$ is not unique: For any $a \in \mathbb{A}$, we have $\mathbb{A} + b = \mathbb{A} + (a+ b)$.
We introduce a simple map that will be important later: the \emph{deaffine map}
\begin{equation}\label{eq:deaff}
\tau : \Graff(k,n)\to \Gr(k,n), \qquad \mathbb{A}+b \mapsto \mathbb{A}
\end{equation}
takes any affine subspace to its corresponding linear subspace.

Let $\mathbb{A}+b \in \Graff(k,n)$. By our notational convention, $\spn(A) = \mathbb{A}$ and therefore the  matrix $[A,b] \in \mathbb{R}^{n \times (k+1)}$ defines the affine subspace  $\mathbb{A}+b $ and we will call this its \emph{affine coordinates}. If in addition, we have  $A \in \V(k,n)$, i.e., an orthonormal basis for $\mathbb{A}$, and we choose $b_0 \in \mathbb{R}^n$ to be orthogonal to $\mathbb{A}$ with $\spn(A) + b_0 = \mathbb{A} + b$, then we call $[A,b_0] \in \V(k,n) \times \mathbb{R}^n$ an \emph{orthogonal affine coordinates} of $\mathbb{A} + b$.
Note that  $A^\tp  A = I$, $A^\tp  b_0 = 0$, and two orthogonal affine coordinates $[A, b_0], [A',b_0'] \in \mathbb{R}^{n \times (k+1)}$  of the same affine subspace $\mathbb{A} + b$ must have that $A' = AQ$ for some $Q \in \O(k)$ and $b_0' = b_0$.

We will also need to discuss the cases where $k = \infty$ and $n = \infty$ as they will be important in Sections~\ref{sec:top} and \ref{sec:metric}. For each $k \in \mathbb{N}$, the infinite flag $\{0\} \subseteq \mathbb{R} \subseteq \mathbb{R}^2 \subseteq \cdots$ induces a directed system
\begin{equation}\label{eqn:direct system}
\cdots \subseteq \Graff(k,n) \subseteq \Graff(k,n+1) \subseteq \cdots,
\end{equation}
and taking direct limit gives
\[
\Graff(k,\infty) \coloneqq \varinjlim \Graff(k,n),
\]
which we will call the \emph{infinite Grassmannian of $k$-dimensional affine linear subspaces} or \emph{infinite affine Grassmannian} for short. This parameterizes $k$-dimensional flats in $\mathbb{R}^n$ for \emph{all} $n \ge k$ and is the affine analogue of the infinite or Sato Grassmannian $\Gr(k, \infty)$ \cite{Sato}.

To be more precise, the direct limit above is taken in the directed system given by the natural inclusions $\iota_n : \Graff(k,n)\to \Graff(k,n+1)$ for $n \ge k$. If $\mathbb{A} + b \in \Graff(k,n)$ has affine coordinates $[A,b] \in \mathbb{R}^{n \times (k+1)}$, then $\iota_n(\mathbb{A} + b) = \mathbb{A}' + b'$ where
$\mathbb{A}' =\spn\begin{bsmallmatrix} A \\ 0 \end{bsmallmatrix}$, $b' = \begin{bsmallmatrix} b \\ 0 \end{bsmallmatrix}$,
i.e., $ \mathbb{A}' + b'\in \Graff(k,n+1)$ has affine coordinates
$\begin{bsmallmatrix} A  & b\\ 0 & 0 \end{bsmallmatrix} \in \mathbb{R}^{(n+1) \times (k+1)}$.
Readers unfamiliar with direct limits may simply identify $[A,b]$ with $\begin{bsmallmatrix} A  & b\\ 0 & 0 \end{bsmallmatrix}$ and thereby regard
\[
\Graff(k,n) \subseteq \Graff(k,n+1)\qquad \text{and} \qquad
\Graff(k,\infty)  = \bigcup\nolimits_{n = k}^\infty \Graff(k,n).
\]
It is straightforward to verify that the deaffine map $\tau: \Graff(k,n) \to \Gr(k,n)$ is compatible with the directed systems $\{\Graff(k,n)\}_{n=k}^{\infty}$ and $\{\Gr(k,n)\}_{n=k}^{\infty}$, i.e., the following diagram commutes:
\begin{equation}\label{eq:comm}
     \begin{tikzcd}
   \cdots \arrow[hookrightarrow]{r} \arrow{d}{\tau}  &  \Graff(k,n)  \arrow[hookrightarrow]{r}{\iota_n}  \arrow{d}{\tau}  & \Graff(k,n+1)   \arrow{d}{\tau}  \arrow[hookrightarrow]{r}  & \cdots \arrow{d}{\tau} \\
   \cdots \arrow[hookrightarrow]{r} &  \Gr(k,n) \arrow[hookrightarrow]{r}  &  \Gr(k,n+1)  \arrow[hookrightarrow]{r}  & \cdots 
     \end{tikzcd}
\end{equation}

Note that one advantage afforded by $\Graff(k,\infty)$ is that one may discuss a $k$-dimensional affine subspace without reference to an ambient space (although strictly speaking, points in  $\Graff(k,\infty)$ are $k$-flats in $\mathbb{R}^\infty \coloneqq \varinjlim \mathbb{R}^n$). The \emph{doubly infinite affine Grassmannian}, which parameterizes affine subspaces of all dimensions, may then be defined as the disjoint union
\[
\Graff(\infty,\infty) \coloneqq \coprod\nolimits_{k=1}^\infty \Graff(k,\infty).
\]
This is the affine analogue of $\Gr(\infty,\infty)$, the doubly infinite Grassmannian of linear subspaces of all dimensions, defined in \cite[Section~5]{YL}.

For the affine Grassmannian, two groups will play the roles that $\O(n)$ and $\GL(n)$ play for the Grassmannian in Section~\ref{sec:basic}. We  defer the discussion to Section~\ref{sec:diff} but will introduce the relevant algebra here. The \emph{group of orthogonal affine transformations} or \emph{orthogonal affine group}, denoted $\E(n)$, is the set $\O(n)\times \mathbb{R}^n$ endowed with group operation
\[
(Q_1, c_1)(Q_2,c_2) = (Q_1Q_2, c_1 + Q_1c_2).
\]
In other words, it is a semidirect product: $\E(n) = \O(n) \ltimes_\vartheta \mathbb{R}^{n}$ where $\vartheta : \O(n) \to \Aut(\mathbb{R}^n) = \GL(n)$ as inclusion. The \emph{group of affine transformations} or \emph{general affine group}, denoted $\GA(n) $, is the set $\GL(n) \times \mathbb{R}^n$ endowed with group operation
\[
(X_1,c_1)  (X_2,c_2) = (X_1X_2, c_1 + X_1c_2).
\]
In other words, it is a semidirect product: $\GA(n) = \GL(n) \ltimes_\iota \mathbb{R}^n$ where $\iota : \GL(n) \to \Aut(\mathbb{R}^n) = \GL(n)$ is the identity map. $\GA(n)$ acts on $\mathbb{R}^n$ naturally via
\[
(X,c) \cdot v = Xv + c,\qquad (X,c) \in \GA(n),\; v\in \mathbb{R}^n.
\]
Clearly $\E(n)$  is a subgroup of $\GA(n)$ and therefore inherits this group action. We note that $\E(n)$ has wide-ranging applications in engineering \cite{CK}.

\section{Differential geometry of the affine Grassmannian}\label{sec:diff}

The affine Grassmannian has rich geometric properties. We start by showing that it is a  noncompact smooth manifold and then show that it is (i) homogeneous, (ii) reductive, and (iii) Riemmannian.
\begin{proposition}\label{prop:smooth}
$\Graff(k,n)$ is a noncompact smooth manifold with
\[
\dim \Graff(k,n) = (n-k)(k+1).
\]
\end{proposition}
\begin{proof}
Let $\mathbb{A}+b\in \Graff(k,n)$ be represented by affine coordinates
$[A,b_0] = [a_1, a_2,\dots a_k, b_0]  \in \mathbb{R}^{n\times (k+1)}$,
where $b_0 $ is chosen so that $ b- b_0\in \mathbb{A}$. Since $A$ has rank $k$, without loss of generality, we may assume that the $k\times k$ leading principal minor of $A$ is nonzero.
 
Let $U$ be the set of all  $\mathbb{X} + y \in \Graff(k,n)$ whose affine coordinates $[X, y_0]$ have nonzero $k \times k$ leading principal minors. Then $U$ is an open subset of $\Graff(k,n)$ containing  $\mathbb{A}+b$. Each $\mathbb{X}+y\in U$ has unique affine coordinates $[\hat{X},\hat{y}]  \in \mathbb{R}^{n\times (k+1)}$ of the form 
\[
[\hat{X},\hat{y}]=
\begin{bsmallmatrix}
1 & 0 &\dots& 0 & 0\\
0&1& \dots &0& 0\\
\vdots& \vdots &\ddots& \vdots & \vdots\\
0 & 0 & \dots &1 & 0\\
\hat{x}_{k+1,1} & \hat{x}_{k+1,2} & \dots & \hat{x}_{k+1,k} & \hat{y}_{k+1}\\
\vdots& \vdots &\ddots& \vdots & \vdots\\
\hat{x}_{n,1} & \hat{x}_{n,2} & \dots & \hat{x}_{n,k} & \hat{y}_{n}\\
\end{bsmallmatrix}.
\]
It is routine to verify that $\varphi: U\to \mathbb{R}^{(n-k)(k+1)}$, $\mathbb{X}+y \mapsto [\hat{X},\hat{y}]$,
is a homeomorphism and thus gives a local chart for $U$. We may likewise define other local charts by the nonvanishing of other $k\times k$ minors and verify that the transition functions $\varphi_1 \circ \varphi_2^{-1}$ are smooth for any two such local charts $\varphi_i : U_i \to \mathbb{R}^{(n-k)(k+1)}$, $i=1,2$. To see the noncompactness, take a sequence in $\Graff(k,n)$ represented in orthogonal affine coordinates by $[A, mb]$ with $m\in \mathbb{N}$, $A = [a_1,\dots, a_k] \in \V(k,n)$, and $0 \ne b \in \mathbb{R}^n$ such that $A^\tp b = 0$; observe that it has no convergent subsequence.
\end{proof}

The \emph{affine Stiefel manifold} is defined to be the product manifold $\Vaff(k,n) \coloneqq \V(k,n) \times \mathbb{R}^{n}$.  It is a homogeneous space because of the following analogue of \eqref{eq:stief},
\[
\Vaff(k,n) \cong \E(n)/\O(n-k)
\]
where $\E(n)$ is the orthogonal affine group $\E(n)$ introduce at the end of Section~\ref{sec:alg}. We have the following characterizations of $\Graff(k,n)$ as quotients of $\E(n)$.
\begin{proposition}\label{prop:homo}
$\Graff(k,n)$ is a reductive homogeneous Riemannian manifold.  In fact, we have the following analogue of \eqref{eq:grass},
\[
\Graff(k,n) \cong \Vaff(k,n)/\E(k) \cong \E(n)/ \bigl(\O(n-k)\times \E(k)\bigr).
\]
Furthermore, $\Vaff(k,n)$ is a principal $\E(k)$-bundle over $\Graff(k,n)$.
\end{proposition}
\begin{proof}
Since $\Graff(k,n)$ can be identified with an open subset of $\Gr(k+1,n+1)$, the Riemannian metric $g_e$ on $\Gr(k+1,n+1)$ induces a metric on $\Graff(k,n)$. Equipped with this induced metric, $\Graff(k,n)$ is a Riemannian manifold. The group $\E(n)$ acts on $\Graff(k,n)$ by 
$(Q,c) \cdot (\mathbb{A}+b)=Q \cdot \mathbb{A}+Qb+c$,
where $(Q,c)\in \E(n)=\O(n)\times \mathbb{R}^n$, $\mathbb{A}+b \in \Graff(k,n)$, and $Q \cdot \mathbb{A} \coloneqq \spn(QA)$. It is easy to see that $\E(n)$ acts on $\Graff(k,n)$ transitively and so
$\Graff(k,n)\cong \E(n)/ \Stab_{\mathbb{A}+b}\bigl(\E(n)\bigr)$,
where $\Stab_{\mathbb{A}+b}\bigl(\E(n)\bigr)$ is the stabilizer of any fixed affine linear subspace $\mathbb{A}+b\in \Graff(k,n)$ in $\E(n)$. Now $\Stab_{\mathbb{A}+b}\bigl(\E(n)\bigr)$ consists of two types of actions. The first action is the affine action inside the plane $\mathbb{A}$, which is $\E(k)$, while the second action is the rotation around the orthogonal complement of $\mathbb{A}$, which is $\O(n-k)$. Hence we obtain
$\Stab_{\mathbb{A}+b}\bigl(\E(n)\bigr)\cong \O(n-k)\times \E(k)$,
and the representation of $\Graff(k,n)$ as a homogeneous Riemannian manifold follows. $\Vaff(k,n)$ is a principal $\E(k)$-bundle over $\Vaff(k,n)/ \E(k) \cong \Graff(k,n)$.
\end{proof}

Let $\tau_v :  \Vaff(k,n) =  \V(k,n) \times \mathbb{R}^{n} \to \V(k,n)$ be the projection. For any $k < n$, $\tau_v$ commutes with the deaffine map $\tau$ in \eqref{eq:deaff}:
\begin{equation}\label{eq:VG}
     \begin{tikzcd}
    \Vaff(k,n)   \arrow{r}{\tau_v}  \arrow{d}{\pi_{a}}  & \V(k,n)  \arrow{d}{\pi}   \\
     \Graff(k,n) \arrow{r}{\tau}  &  \Gr(k,n)  
     \end{tikzcd}
\end{equation}
where we view $\Graff(k,n)$, $\Gr(k,n)$, $\Vaff(k,n)$, $\V(k,n)$ as  homogeneous spaces. One may define $\V(k,\infty)$, the Stiefel manifold of orthogonal $k$-frames in $\mathbb{R}^\infty$, as the direct limit of  the inclusions $\iota_n: \V(k,n) \to \V(k,n+1)$, $Q \mapsto \begin{bsmallmatrix}Q \\ 0 \end{bsmallmatrix}$, 
and its affine counterpart as $\Vaff(k,\infty) \coloneqq \V(k,\infty) \times \mathbb{R}^\infty$, the \emph{infinite affine Stiefel manifold}. Taking direct limit of \eqref{eq:VG}, we obtain
\begin{equation}\label{diagram:infinite Stiefel}
     \begin{tikzcd}
    \Vaff(k,\infty)  \arrow{r}{\tau_v}  \arrow{d}{\pi_{a}}  & \V(k,\infty)   \arrow{d}{\pi}   \\
     \Graff(k,\infty) \arrow{r}{\tau}  &  \Gr(k,\infty)  
     \end{tikzcd}
\end{equation}
The objects in \eqref{diagram:infinite Stiefel} are all Hilbert manifolds although we will not use this fact.

From a computational perspective, one would prefer to work with orthogonal objects like $\V(k,n)$ and $\O(k)$ rather than affine objects like $\Vaff(k,n)$ and $\E(k)$. Roughly speaking, this is largely because orthogonal transformations preserve norm and do not magnify rounding errors during computations. With this in mind, we will seek to characterize the  affine Grassmannian as an orbit space of the orthogonal group in a Stiefel manifold.

Let $\mathbb{A} + b \in\Graff(k,n)$. Its orthogonal affine coordinates are $[A,b_0] \in \V(k,n) \times \mathbb{R}^n$ where $A^\tp  b_0 = 0$, i.e., $b_0$ is orthogonal to the columns of $A$. However as $b_0$ is in general not of unit norm, we may not regard $[A,b_0]$ as an element of $\V(k+1,n)$. The following variant\footnote{Definition~\ref{def:stief} has appeared in \cite[Definition~3.1]{WYL}. We reproduce it here for the reader's easy reference.} is a convenient system of coordinates for computations \cite{WYL} and for defining various distances on $\Graff(k,n)$ in Section~\ref{sec:metric}.
\begin{definition}\label{def:stief}
Let $\mathbb{A} + b \in \Graff(k,n)$  and $[A,b_0] \in \mathbb{R}^{n \times (k+1)}$ be its orthogonal affine coordinates, i.e., $A^\tp  A = I$ and $A^\tp  b_0 =0$. The matrix of \emph{Stiefel coordinates} for $\mathbb{A} + b$ is the $(n+1) \times (k+1)$ matrix with orthonormal columns,
\[
Y_{\mathbb{A} + b} \coloneqq
\begin{bmatrix}
A& b_0/\sqrt{1+\lVert b_0\rVert^2}\\
0& 1/\sqrt{1+\lVert b_0\rVert^2}
\end{bmatrix} \in \V(k+1, n +1).
\]
\end{definition}
Two orthogonal affine coordinates $[A,b_0], [A',b_0']$ of $\mathbb{A}+b$ give two corresponding matrices of Stiefel coordinates $Y_{\mathbb{A} + b} $, $Y_{\mathbb{A} + b}'$. By the remark after our definition of orthogonal affine coordinates, $A = A'Q'$ for some $Q' \in \O(k)$ and $b_0 = b_0'$. Hence
\begin{equation}\label{eq:otsc}
Y_{\mathbb{A} + b} =
\begin{bmatrix}
A& b_0/\sqrt{1+\lVert b_0\rVert^2}\\
0& 1/\sqrt{1+\lVert b_0\rVert^2}
\end{bmatrix}= 
\begin{bmatrix}
A' & b_0'/\sqrt{1+\lVert b_0'\rVert^2}\\
0& 1/\sqrt{1+\lVert b_0'\rVert^2}
\end{bmatrix}
\begin{bmatrix}
Q' & 0 \\
0 & 1
\end{bmatrix} =
Y_{\mathbb{A} + b}' Q
\end{equation}
where $Q \coloneqq \begin{bsmallmatrix} Q' & 0 \\ 0 &1\end{bsmallmatrix}\in \O(k+1)$. Hence two different matrices of Stiefel coordinates for the same affine subspace differ by an orthogonal transformation.

There is also an affine counterpart to the last paragraph of Section~\ref{sec:basic} that allows us to provide an analogue of Proposition~\ref{prop:homo} without reference to orthogonality, useful for studying the affine Grassmannian over a vector space without an inner product.
The \emph{noncompact affine Stiefel manifold} $\Staff(k,n)$ may be defined in several ways:
\[
\Staff(k,n) = \GA(n)/\GL(n-k) = \bigl(\GL(n) /\GL(n-k) \bigr) \times \mathbb{R}^n = \St(k,n) \times \mathbb{R}^n,
\]
where $\GA(n)$ is the general affine group in Section~\ref{sec:alg} and $\St(k,n)$ the noncompact Stiefel manifold in Section~\ref{sec:basic}.
\begin{proposition}
\begin{enumerate}[\upshape (i)]
\item\label{it:affdim} Dimensions of the compact and noncompact affine Stiefel manifolds are
\[
\dim \Vaff(k,n) = \frac{1}{2}k( 2n - k + 1), \qquad \dim \Staff(k,n) = n (k+1).
\] 
\item\label{it:affiso} Whether as topological spaces, differential manifolds, or algebraic varieties, we have
\[
\Graff(k,n) \cong \Staff(k,n)/\GA(k) \cong \GA(n)/\bigl( \GL(n-k) \times  \GA(k) \bigr),
\]
i.e., the isomorphism is a homeomorphism, diffeomorphism, and biregular map.
\item\label{it:affbun} $ \Staff(k,n)$ is a principal $\GA(k)$-bundle over $\Graff(k,n)$.
\end{enumerate}
\end{proposition}
\begin{proof}
 The inclusion $\E(n)\hookrightarrow \GA(n)$ as a subgroup naturally induces the  commutative diagram:
\begin{equation}\label{diag:noncpt Stiefel}
     \begin{tikzcd}
  \E(n) \arrow[hookrightarrow]{r}    \arrow{d}{} & \GA(n)  \arrow{d}{}\  \\
  \Vaff(k,n) = \E(n)/\O(n-k) \arrow[hookrightarrow]{r}\arrow{d}{\pi_a} &  \Staff(k,n)\arrow{d}{\pi_s} = \GA(n)/\GL(n-k) \\
  \Graff(k,n) = \Vaff(k,n)/\E(k) \arrow[hookrightarrow]{r}{j} \arrow{d}{\tau} & \Staff(k,n)/ \GA(k) \arrow{d}{\tau_s} \\
\Gr(k,n)  \ \arrow{r}{\cong} & \GL(n)/\bigl(\GL(n-k) \times \GL(k) \bigr)
     \end{tikzcd}
\end{equation}
where $\pi_a$ and $\tau$ are as in \eqref{eq:VG}, $\pi_s$ is the quotient map, and $\tau_s$ is similarly defined as  $\tau$. The bottom isomorphism is \eqref{eq:ncgrass}, which is simultaneously an isomorphism of topological spaces, differential manifolds, and algebraic varieties. 
\eqref{it:affdim} follows from the quotient space structures:
$\dim \Vaff(k,n) = \dim \E(n) - \dim \O(n-k)$ and $\dim \Staff(k,n) = \dim \GA(n) - \dim \GL(n-k)$.
For \eqref{it:affiso}, it suffices to show that the restriction
\[
j|_{\tau^{-1}(\mathbb{A})}:\tau^{-1}(\mathbb{A}) \to \tau_s^{-1}(\mathbb{A})
\]
is an isomorphism for every $\mathbb{A} \in \Gr(k,n)$, but this follows from the bottom isomorphism. \eqref{it:affbun} follows from \eqref{it:affiso} as $\Staff(k,n)$ is a principal $\GA(k)$-bundle on $\Staff(k,n)/\GA(k) \cong \Graff(k,n)$.
\end{proof}

\section{Algebraic geometry of the affine Grassmannian}\label{sec:AG}

We now turn to the algebraic geometric aspects, characterizing $\Graff(k,n)$ as (i) an irreducible nonsingular algebraic variety, (ii) an Zariski open dense subset of $\Gr(k+1,n+1)$, and (iii) a real affine algebraic variety of projection matrices. In addition, just as  $\Gr(k,n)$ is a moduli space of $k$-dimensional linear subspaces in $\mathbb{R}^n$, $\Graff(k,n)$ is a moduli space of $k$-dimensional  affine subspaces in $\mathbb{R}^n$, although we have nothing to add beyond this observation. In Section~\ref{sec:schubert}, we will discuss affine Schubert varieties, an analogue of Schubert varieties, in $\Graff(k,n)$.

That $\Graff(k,n)$ may be regarded as a Zariski dense subset of  $\Gr(k+1,n+1)$ is a noteworthy point. It is the key to our optimization algorithms in \cite{WYL}. Also, it immediately implies that any probability densities \cite{CY} defined on the usual Grassmannian may be adapted to the affine Grassmannian, a fact that we will rely on in Section~\ref{sec:prob}. 
\begin{theorem}\label{thm:alg}
\begin{enumerate}[\upshape (i)]
\item $\Graff(k,n)$ is an algebraic variety that is irreducible and nonsingular.
\item $\Graff(k,n)$ may be embedded as a Zariski open subset of $\Gr(k+1,n+1)$,
\begin{equation}\label{eq:j}
j : \Graff(k,n) \to \Gr(k+1,n+1), \quad \mathbb{A}+b \mapsto \spn(\mathbb{A}\cup \{ b +e_{n+1} \} ),
\end{equation}
where $e_{n+1} = (0,\dots,0,1)^\tp \in \mathbb{R}^{n+1}$. The image is open and dense in both the Zariski and manifold topologies.

\item $\Gr(k+1,n+1)$ may be regarded as the disjoint union of $\Gr(k+1,n)$ and $\Graff(k,n)$; more precisely,
\[
\Gr(k+1,n+1) =  X \cup X^c, \quad X \cong \Graff(k,n), \quad X^c \cong \Gr(k+1,n).
\]
\end{enumerate}
\end{theorem}
\begin{proof}
Substituting `smooth' with `regular' and `differential manifold' by `algebraic variety' in  the proof of Proposition~\ref{prop:smooth}, we see that $\Graff(k,n)$ is a nonsingular algebraic variety. Its irreducibility follows from  Theorem~\ref{thm:bundle} since $\Gr(k,n)$ is irreducible and all fibers of $\Graff(k,n)\to \Gr(k,n)$ are irreducible and of the same dimension. We use `algebraic variety' is used here in the sense of an abstract algebraic variety, i.e., $\Graff(k,n)$ is obtained by gluing together affine open subsets. 

The embedding $j$ takes $k$-flats in $\mathbb{R}^n$ to $(k+1)$-planes in $\mathbb{R}^{n+1}$, i.e.,
$\mathbb{R}^n \supseteq \mathbb{A}+b \mapsto \spn(\mathbb{A}\cup \{ b +e_{n+1} \} )\subseteq\mathbb{R}^{n+1}$.
It maps $\mathbb{R}^n$ onto $E_n \coloneqq \spn\{e_1,\dots,e_n\} \subseteq \mathbb{R}^{n+1}$ where $e_1,\dots,e_n, e_{n+1}$ are the standard basis vectors of $\mathbb{R}^{n+1}$. Linear subspaces $\mathbb{A} \subseteq \mathbb{R}^n$ are then mapped to $j(\mathbb{A}) \subseteq E_n$. Clearly $j$ is an  embedding. We  illustrate  this embedding with the case $k = 1$, $n = 3$  in Figure~\ref{fig:plane}.
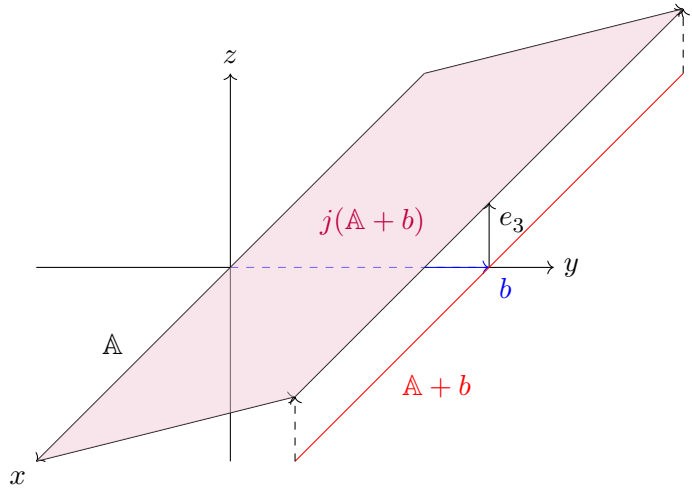
\begin{figure}[!ht]
\centering
\begin{tikzpicture}[scale=0.86]
\draw[black] (-3,0)--(0,0);
\draw[dashed,blue] (0,0)--(4,0);
\draw[->] (4,0) -- (5,0)node[anchor=west] {$y$};
\draw[->] (0,-3) -- (0,3)node[anchor=south] {$z$};
\node [above left,black] at (-1.5,-1.5) {$\mathbb{A}$}; 
\draw[black][->] (3,3) -- (-3,-3);
\node [below left,black] at (-3,-3) {$x$};
\draw[blue][->] (3,0)--(4,0);
\node [below right,blue] at (4,0) {$b$}; 
\draw[black][->]  (4,0)-- (4,1);
\node [below right,black] at (4,1) {$e_3$}; 
\draw[red] (1,-3) -- (7,3);
\node [below right,red] at (2.5,-1.5) {$\mathbb{A}+b$}; 
\draw[dashed][->]  (1,-3)-- (1,-2);
\draw[dashed][->]  (7,3)-- (7,4);
\draw[black][->] (1,-2) -- (7,4);
\draw[black][->] (-3,-3)--(1,-2);
\draw[black][->] (3,3)-- (7,4);
\fill[purple!20,opacity=0.5] (-3,-3) -- (1,-2) -- (7,4) -- (3,3) -- cycle;
\node[purple] at (2.2,0.7) {$j(\mathbb{A}+b)$};
\end{tikzpicture} 
\caption{Here our linear subspace $\mathbb{A}$ is the $x$-axis. It is displaced by $b$ along the $y$-axis to the affine subspace $\mathbb{A}+b$. The embedding  $j : \Graff(k,n) \to \Gr(k+1,n+1)$ takes $\mathbb{A} + b$ to the smallest $2$-plane containing $\mathbb{A}$ and $b+e_3$, where $e_3$ is a unit vector along the $z$-axis.}
\label{fig:plane}
\end{figure}


We set $X \coloneqq j\bigl(\Graff(k,n)\bigr) \subseteq \Gr(k+1,n+1)$ and set $X^c$ to be the set-theoretic complement of $X$ in $\Gr(k+1,n+1)$. By (ii), $X \cong \Graff(k,n)$. By the definition of $X^c$, a $(k+1)$-plane $\mathbb{B}\in \Gr(k+1,n+1)$ is in $X^c$ if and only if $\mathbb{B}\subseteq E_n$, which is to say that
$X^c = \Gr_{k+1}(E_n) \cong \Gr(k+1,n)$. 
Lastly we see that $X$ is Zariski open because its complement $X^c$, comprising $(k+1)$-planes in $E_n$, is clearly Zariski closed.
\end{proof}

Henceforth we will identify
\begin{equation}\label{eq:Rn}
\mathbb{R}^n \equiv \{ (x_1,\dots,x_n,0)^\tp  \in \mathbb{R}^{n+1} : x_1,\dots,x_n \in \mathbb{R}\}
\end{equation}
to obtain a complete flag
\[
\{0\} \subseteq \mathbb{R}^1 \subseteq \mathbb{R}^2 \subseteq \dots \subseteq \mathbb{R}^n \subseteq \mathbb{R}^{n+1} \subseteq \cdots,
\]
which was essentially what we did in the proof of Theorem~\ref{thm:alg}. With such an identification, our choice of $e_{n+1}$ in the embedding $j$ in \eqref{eq:j} is the most natural one.

It is often desirable to uniquely represent elements of $\Graff(k,n)$ as actual matrices instead of equivalence classes of matrices like the affine, orthogonal affine, and Stiefel coordinate representations in Sections~\ref{sec:alg} and \ref{sec:diff}. For example, we will see that this is the case when we discuss probability distributions on  $\Graff(k,n)$ in Section~\ref{sec:prob}. The Grassmannian has a well-known representation  \cite[Example~1.2.20]{Nicolaescu}  as the set of rank-$k$ orthogonal projection\footnote{A projection matrix satisfies $P^2 = P$ and an orthogonal projection matrix is in addition symmetric, i.e., $P^\tp  = P$. Despite its name, an orthogonal projection matrix $P$ is not an orthogonal matrix unless $P=I$.} matrices, or, equivalently, the set of trace-$k$ idempotent symmetric matrices:
\begin{equation}\label{eq:grproj}
\Gr(k,n) \cong \{P \in \mathbb{R}^{n \times n}: P^\tp  = P^{2} = P, \; \tr(P) = k\}.
\end{equation}
Note that $\rank(P) = \tr(P)$ for an orthogonal projection matrix $P$.
A straightforward affine analogue of \eqref{eq:grproj} for $\Graff(k,n)$ is the following.
\begin{proposition}\label{prop:rav}
$\Graff(k,n)$ is a real affine algebraic variety given by
\begin{equation}\label{eq:graffmatrix}
\Graff(k,n) \cong \{[P , b] \in \mathbb{R}^{n \times (n+1)}:  P^\tp =P^2=P, \; \tr(P)=k,\; Pb = 0\}.
\end{equation}
\end{proposition}
\begin{proof}
Let $\mathbb{A}+b \in \Graff(k,n)$ have  orthogonal affine coordinates $[A,b_0] \in \mathbb{R}^{n \times (k+1)}$. Recall that if $A$ is an orthonormal basis for the subspace $\mathbb{A}$, then $AA^\tp $ is the orthogonal projection onto $\mathbb{A}$. It is straightforward to check that the map $\mathbb{A}+b \mapsto [AA^\tp , b_0]$ is independent of the choice of orthogonal affine coordinates and is bijective.
\end{proof}

We will call  the matrix $[P , b] \in \mathbb{R}^{n \times (n+1)}$ \emph{projection affine coordinates} for $\mathbb{A} + b$. From a practical standpoint, we would like to represent points in $\Graff(k,n)$ as orthogonal projection matrices; one reason is that such a coordinate system facilitates optimization algorithms on $\Graff(k,n)$  (see \cite{WYL}), another is that certain probability densities can be naturally expressed in such a coordinate system (see Section~\ref{sec:prob}). Since $[P , b]$ is not an orthogonal projection matrix, we introduce the following variant.\footnote{Definition~\ref{def:proj} has appeared in \cite[Definition~3.4]{WYL}. We reproduce it here for the reader's easy reference.}
\begin{definition}\label{def:proj}
Let $\mathbb{A} + b \in \Graff(k,n)$  and $[P , b] \in \mathbb{R}^{n \times (n+1)}$  be its projection affine coordinates. The matrix of \emph{projection coordinates} for $\mathbb{A} + b$ is the  orthogonal projection matrix
\[
P_{\mathbb{A} + b} \coloneqq \begin{bmatrix} P+bb^\tp /(\lVert b \rVert^{2}+1)& b/(\lVert b\rVert^{2}+1) \\ b^\tp /(\lVert b\rvert^{2}+1)& 1/(\lVert b\rVert^{2}+1) \end{bmatrix} \in \mathbb{R}^{(n+1) \times (n+1)}.
\]
Alternatively, in terms of orthogonal affine coordinates $[A,b_0] \in \mathbb{R}^{n \times (k+1)}$,
\[
P_{\mathbb{A} + b} = \begin{bmatrix} AA^\tp  +b_0b_0^\tp /(\lVert b_0 \rVert^{2}+1)& b_0/(\lVert b_0\rVert^{2}+1) \\ b_0^\tp /(\lVert b_0\rvert^{2}+1)& 1/(\lVert b_0\rVert^{2}+1)\end{bmatrix} \in \mathbb{R}^{(n+1) \times (n+1)}.
\]
\end{definition}
It is easy to check that $P_{\mathbb{A} + b}$ is indeed an orthogonal projection matrix, i.e.,
$P_{\mathbb{A} + b}^2 =  P_{\mathbb{A} + b} = P_{\mathbb{A} + b}^\tp $.
Unlike Stiefel coordinates, projection coordinates of a given affine subspace are unique.

\section{Schubert calculus on the affine Grassmannian}\label{sec:schubert}

We will show that basic aspects of Schubert calculus on the Grassmannian \cite{KL1972} could be readily extended to an ``affine Schubert calculus'' on the affine Grassmannian, with affine analogues of flags, Schubert varieties, Schubert cycles \cite{Sottile}. As is the case for (co)homology of the Grassmannian, the materials in this section will be important for our (co)homology calculations in Section~\ref{sec:coho}; what is perhaps more surprising is that our study of distances between affine subspaces of different dimensions in Section~\ref{sec:distdifferent} will also rely on affine Schubert varieties.

In this paragraph, we briefly review some basic terminologies and facts in Schubert calculus for the reader's easy reference. The \emph{Schubert variety} of a flag $\mathbb{A}_1 \subseteq \cdots \subseteq \mathbb{A}_k$ in $\mathbb{R}^n$ is a subvariety of $\Gr(k,n)$ defined by
\[
\Omega(\mathbb{A}_1,\dots, \mathbb{A}_k) \coloneqq \{ \mathbb{B} \in \Gr(k,n) : \dim (\mathbb{B} \cap \mathbb{A}_j) \ge j,\; j = 1,\dots,k \}.
\]
It is a standard fact  \cite[Proposition~4]{KL1972} that
\begin{equation}\label{eq:iso}
\Omega(\mathbb{A}_1,\dots, \mathbb{A}_k) \cong \Omega(\mathbb{B}_1,\dots, \mathbb{B}_k)\quad
 \text{if } \dim \mathbb{A}_j = \dim \mathbb{B}_j, \; j=1,\dots, k.
\end{equation}
So when the choice of the flag is unimportant, we may take it to be 
$\mathbb{R}^{d_1} \subseteq \cdots \subseteq \mathbb{R}^{d_k}$
and denote the corresponding Schubert variety  by $\Omega(d_1,\dots, d_k)$. The `$\cong$' in \eqref{eq:iso} may be taken to either homeomorphism of topological spaces or biregular isomorphism of algebraic variety but it cannot in general be replaced by `$=$'  --- two different flags of the same dimensions determine different varieties in $ \Gr(k,n)$.
The following properties \cite{HHD1994,Prasolov2007} of Schubert varieties are also well-known.
\begin{facts}\label{fact:schubert}
\begin{enumerate}[\upshape (i)]
\item\label{it:dim} The dimension of a Schubert variety is given by
\[
\dim \Omega(d_1,\dots, d_k) = \sum_{j=1}^k d_j - \frac{1}{2} k (k+1).
\]
\item\label{it:cyc} The cycles determined by Schubert varieties of dimension $i$ form a basis for the $i$th homology group $\H_i (\Gr(k,n),\mathbb{Z}_2)$ and cohomology group $\H^i(\Gr(k,n),\mathbb{Z}_2)$, which are isomorphic and
\[
\H_i (\Gr(k,n),\mathbb{Z}_2)  \simeq \H^i(\Gr(k,n),\mathbb{Z}_2) \simeq \mathbb{Z}_2^{r_i},
\]
where $r_i$ is the number of Schubert varieties of dimension $i$. Here $\mathbb{Z}_2 \coloneqq \mathbb{Z}/2\mathbb{Z}$.
\item\label{it:cell} The collection of Schubert varieties in $\Gr(k,n)$ over all flags of length $k$ in $\mathbb{R}^n$ gives a cell decomposition for $\Gr(k,n)$.
\end{enumerate}
\end{facts}

We will now introduce an affine analogue of the Schubert variety in the affine Grassmannian using an \emph{affine flag}, i.e., an increasing sequence of nested affine subspaces.
\begin{definition}\label{def:Psi}
Let $\mathbb{A}_1 + b_1 \subseteq \cdots \subseteq \mathbb{A}_k + b_k$ be an affine flag in $\mathbb{R}^n$. 
The corresponding \emph{affine Schubert variety} is a subvariety of $\Graff(k,n)$ defined by
\[
\Psi(\mathbb{A}_1+b_1,\dots, \mathbb{A}_k+b_k)  \coloneqq \{ \mathbb{B} + c\in \Graff(k,n): \dim \bigl( (\mathbb{B} +c) \cap (\mathbb{A}_j + b_j)\bigr) \ge j,\; j = 1,\dots,k \}.
\]
\end{definition}
We first show that the affine flag may always be chosen such that $b_1 = \dots = b_k$.
\begin{lemma}
For any affine flag $\mathbb{A}_1 + b_1 \subseteq \cdots \subseteq \mathbb{A}_k + b_k$, there exists a displacement vector $b\in \mathbb{R}^n$ such that $\mathbb{A}_j + b_j = \mathbb{A}_j + b$, $j=1,\dots, k$. Thus every affine Schubert variety is of the form $\Psi(\mathbb{A}_1 + b,\dots, \mathbb{A}_k + b)$.
\end{lemma}
\begin{proof}
Let $-b$ be any element in $\mathbb{A}_1 + b_1$. By definition, $-b\in \mathbb{A}_j + b_j$, $j=1,\dots, k$. So $b_j \in \mathbb{A}_j + b$. Therefore $ \mathbb{A}_j  + b_j = \mathbb{A}_j + b$, $j=1,\dots, k$.
\end{proof}
It is straightforward to derive an analogue of \eqref{eq:iso}.
\begin{proposition}\label{prop:transit}
For any two affine flags $\mathbb{A}_1 + b \subseteq \cdots \subseteq \mathbb{A}_k + b$ and $\mathbb{B}_1 + c \subseteq \cdots \subseteq \mathbb{B}_k + c$ where $\dim \mathbb{A}_j  = d_j = \dim \mathbb{B}_j$, $j=1,\dots, k$, we have 
\[
\Psi(\mathbb{A}_1+b,\dots,\mathbb{A}_k + b) \cong \Psi(\mathbb{B}_1 + c, \dots, \mathbb{B}_k + c),
\]
and so we may write $\Psi(d_1,\dots, d_k) $  when the specific affine flag is unimportant. 
\end{proposition}
\begin{proof}
There is a general affine transformation $(X,y)\in \GA(n)$ such that
\[
(X,y)\cdot (\mathbb{A}_j + b) = X(\mathbb{A}_j) + Xb + y  = \mathbb{B}_j + c,\quad j=1,\dots, k,
\]
i.e., $X(\mathbb{A}_j) = \mathbb{B}_j$ and $Xb + y = c$. The existence of $X \in \GL(n)$ is guaranteed by the transitive action of $\GL(n) $ on (linear) flags of fixed dimensions $(d_1,\dots,d_k)$. We then set $y \coloneqq c - Xb$.
\end{proof}

We also provide an analogue of Fact~\ref{fact:schubert}\eqref{it:dim}, whose proof is somewhat more involved.
\begin{theorem}\label{thm:dim}
The dimension of an affine Schubert variety is 
\[
\dim \Psi(d_1,\dots, d_k) = \sum_{j=1}^k d_j - \frac{k(k+1)}{2} + (d_1-1).
\]
\end{theorem}
\begin{proof}
Let $j$ be the embedding in \eqref{eq:j}. We will determine the dimension of $j\bigl(\Psi(d_1,\dots, d_k)\bigr)$, which is clearly the same as that of $\Psi(d_1,\dots, d_k)$. We claim that 
\begin{equation}\label{eqn:dim}
j\bigl(\Psi(d_1,\dots, d_k)\bigr) = \Omega(d_1,d_1+1,\dots, d_k+1) \cap j\bigl(\Graff(k,n)\bigr).
\end{equation}
Since $\Omega(d_1,d_1+1,\dots, d_k+1)$ is an irreducible subset of $\Gr(k+1,n+1)$ and $j\bigl(\Graff(k,n)\bigr)$ is an affine open subset of $\Gr(k+1,n+1)$, we obtain the required dimension via
\begin{align*}
\dim j\bigl(\Psi(d_1,\dots, d_k)\bigr) &= \dim \Omega(d_1,d_1+1,\dots, d_k+1)  \\
&= d_1 + \sum_{i=1}^k (d_i + 1) - \frac{(k+1)(k+2)}{2}  = \sum_{i=1}^k d_i - \frac{k(k+1)}{2} + (d_1-1),
\end{align*}
where we have used Fact~\ref{fact:schubert}\eqref{it:dim} for the second equality.

It remains to establish \eqref{eqn:dim}. Let $\mathbb{A}_0+ b \subseteq \mathbb{A}_1+ b \subseteq \dots \subseteq \mathbb{A}_k+b$ be an affine flag with $\dim (\mathbb{A}_i + b) = d_i$, $i=0,1,\dots, k$ where we set $d_0 \coloneqq d_1 - 1$.

Let $\mathbb{B} + c \in \Psi(\mathbb{A}_1 + b,\dots, \mathbb{A}_k + b)$. As
$\dim (\mathbb{A}_1 + b) \cap (\mathbb{B} + c) \ge 1$,
there is some $-x$ contained in both $\mathbb{A}_1 + b$ and $\mathbb{B} + c$. Since $\mathbb{A}_1 + b \subseteq \cdots \subseteq  \mathbb{A}_k +b$, we have
\begin{equation}\label{eq:x}
\mathbb{A}_i + b = \mathbb{A}_i + x,\quad \mathbb{B} + c = \mathbb{B} + x,\qquad i=1,\dots,k.
\end{equation}
Therefore, for any $i=1,\dots,k$,
\[
\dim j(\mathbb{B} + c) \cap j(\mathbb{A}_i + b) = \dim j(\mathbb{B} + x) \cap j(\mathbb{A}_i + x) 
= \dim (\mathbb{B} + c) \cap (\mathbb{A}_i + b) + 1 
 \ge i + 1.
\]
Since $j(\mathbb{A}_0 + b)$ is a codimension-one linear subspace of $j(\mathbb{A}_1 + b)$, we also have 
\[
\dim j(\mathbb{B} + c) \cap j(\mathbb{A}_0 + b)\ge \dim j(\mathbb{B} + c) \cap j(\mathbb{A}_1 + b) - 1 \ge 1.
\]
Hence we must have
\[
j(\mathbb{B} + c) \in \Omega\bigl( j(\mathbb{A}_0 + b), j(\mathbb{A}_1 + b),\dots, j(\mathbb{A}_k + b)\bigr) \cap j\bigl(\Graff(k,n)\bigr).
\]
This shows the ``$\subseteq$'' in \eqref{eqn:dim}.

Conversely, let $j(\mathbb{B} + c) \in \Omega\bigl( j(\mathbb{A}_0 + b),j(\mathbb{A}_1 + b),\dots, j(\mathbb{A}_k + b)\bigr)$. On the one hand, we have
\[
\dim j(\mathbb{B} + c) \cap  j(\mathbb{A}_i + b) \ge i + 1;
\]
and on the other hand, since $j$ is an embedding, we have 
\[
j \bigl((\mathbb{B} + c ) \cap (\mathbb{A}_i + b ) \bigr) = j(\mathbb{B} + c) \cap  j(\mathbb{A}_i +b),
\]
for any $i=0,1,\dots, k$.
Therefore, we have
\[
\dim ( \mathbb{B} + c ) \cap (\mathbb{A}_i + b ) = \dim j(\mathbb{B} + c) \cap  j(\mathbb{A}_i + b) \ge i + 1,\quad i=0,1,\dots, k.
\]
In other words,  $\mathbb{B} + c \in \Psi(\mathbb{A}_1 + b,\dots,\mathbb{A}_k + b)$. This shows the ``$\supseteq$'' in \eqref{eqn:dim}.
\end{proof}

In Section~\ref{sec:coho}, we will give the affine analogues of Facts~\ref{fact:schubert}\eqref{it:cyc} and \eqref{it:cell} as Theorem~\ref{thm:homology} and Proposition~\ref{prop:cell} respectively.

There are two affine Schubert varieties that deserve special mention because of their importance in our metric geometry discussions in Section~\ref{sec:distdifferent} and, to a lesser extent, also the probability discussions in Section~\ref{sec:prob}.
\begin{definition}\label{def:Psi+-}
Let $\mathbb{A}+b \in \Graff(k,n)$ and $\mathbb{B}+c \in \Graff(l,n)$ where $k \le l \le n$. The \emph{affine Schubert varieties} of $l$-flats containing $\mathbb{A}+b$ and $k$-flats contained in $\mathbb{B} + c$ are respectively
\begin{equation}\label{eq:asv}
\begin{aligned}
\Psi_{+}(\mathbb{A}+b)&\coloneqq \bigl\{\mathbb{X}+y\in \Graff(l,n) : \mathbb{A}+b\subseteq \mathbb{X}+y\bigr\}, \\
\Psi_{-}(\mathbb{B}+c)&\coloneqq \bigl\{\mathbb{Y}+z\in \Graff(k,n) : \mathbb{Y}+z\subseteq \mathbb{B}+c\bigr\}.
\end{aligned}
\end{equation}
\end{definition}
The nomenclature in Definition~\ref{def:Psi+-} is justified as $\Psi_{+}(\mathbb{A}+b)$ is the affine Schubert variety of the affine flag 
\begin{equation}\label{eqn:flagdescription1}
\{0\} \eqqcolon \mathbb{A}_0 + b_0 \subseteq \mathbb{A}_1 + b_1 \subseteq \dots \subseteq \mathbb{A}_k + b_k \coloneqq \mathbb{A} + b \subseteq \dots \subseteq \mathbb{A}_l + b_l,
\end{equation}
where $\mathbb{A}_{k+i} + b_{k+i}$ is an affine subspace of dimension $n-l+(k+i)$, $i =1,\dots,l-k$; and
$\Psi_{-}(\mathbb{B}+c)$ is the affine Schubert variety of the affine flag
\begin{equation}\label{eqn:flagdescription2}
\{0\} \eqqcolon \mathbb{B}_0 + c_0 \subseteq \mathbb{B}_1 + c_1 \subseteq \dots  \subseteq \mathbb{B}_k + c_k \coloneqq \mathbb{B} + c
\end{equation}
where $\mathbb{B}_j + c_j$ is an affine subspace of dimension $l-k+j,j=1,\dots,k$.

We next discuss the geometry of these sets, starting with the observation that 
$\Psi_{+}(\mathbb{A}+b) $ is isomorphic to a Grassmannian and $\Psi_{-}(\mathbb{B}+c)$ is isomorphic to an affine Grassmannian. 
\begin{proposition}\label{prop:OmegaGraff}
Let $\mathbb{A}+b\in \Graff(k,n)$ and $\mathbb{B}+c\in \Graff(l,n)$.  Then
\[
\Psi_{+}(\mathbb{A}+b) \cong \Gr(n-l,n-k) \quad \emph{and} \quad
\Psi_{-}(\mathbb{B}+c) \cong \Graff(k,l) 
\]
as Riemannian manifolds and algebraic varieties. In particular, we have 
\[
\dim \Psi_+(\mathbb{A} + b) = (n - l)(l-k),\quad \dim \Psi_-(\mathbb{B} + c) = (k+1)(l-k).
\]
\end{proposition}
\begin{proof}
We first observe that the map
$\varphi:\Psi_{+}(\mathbb{A}+b)\to \Omega_{+}(\mathbb{A})$,
$\mathbb{X}+y \mapsto \mathbb{X}+y-b$,
is well-defined since $\mathbb{A}\subseteq \mathbb{X}+y-b$ by our choice of $\mathbb{X}+y$. Also,
$\psi:\Omega_{+}(\mathbb{A}) \to \Psi_{+}(\mathbb{A}+b)$,
$\mathbb{X}\mapsto \mathbb{X}+b$,
is the inverse of $\varphi$ and so it is an isomorphism. Together with \cite[Proposition~21]{YL}, we obtain the first isomorphism $\Psi_{+}(\mathbb{A}+b)\cong \Omega_{+}(\mathbb{A})\cong \Gr(n-l,n-k)$.
For the second isomorphism, consider
$\varphi': \Psi_{-}(\mathbb{B}+c)\to \Graff_k(\mathbb{B})$,
$\mathbb{Y}+z\mapsto \mathbb{Y}+z-c$,
which is well-defined since $\mathbb{Y}+z-c$ is an affine subspace of dimension $k$ in $\mathbb{B}$. Its inverse is given by
$\psi':\Graff_k(\mathbb{B}) \to \Psi_{-}(\mathbb{B}+c)$,
$\mathbb{Y}+z\mapsto \mathbb{Y}+z+c$,
and so it is an isomorphism. The required isomorphism then follows from $\Psi_{-}(\mathbb{B}+c) \cong \Graff_k(\mathbb{B}) \cong \Graff(k,l)$.
\end{proof}
The asymmetry in Proposition~\ref{prop:OmegaGraff} is expected. $\Psi_{+}(\mathbb{A}+b) $ is a Grassmannian of \emph{linear} subspaces since all affine subspaces containing $\mathbb{A}+b$ can be shifted back to the origin by the vector $b$. In the case of $\Psi_{-}(\mathbb{B}+c)$, shifting $\mathbb{B}+c$ back to the origin by $c$ and then taking all affine subspaces contained in $\mathbb{B}$ still gives a Grassmannian of \emph{affine} subspaces. As a sanity check, note that the dimensions in Proposition~\ref{prop:OmegaGraff} agree with their values given by  Theorem~\ref{thm:dim}  with respect to the affine flags \eqref{eqn:flagdescription1} and \eqref{eqn:flagdescription2}.

We also have the following analogue of Proposition~\ref{prop:rav} that allows us to regard $\Psi_{+}(\mathbb{A} + b)$, $\Psi_{-}(\mathbb{B} + c)$ as subsets of $n \times (n+1)$ matrices.
\begin{proposition}
The affine Schubert varieties $\Psi_{+}(\mathbb{A} + b)$ and $\Psi_{-}(\mathbb{B} + c)$ are isomorphic to real affine algebraic varieties in $ \mathbb{R}^{n \times (n+1)}$ given by
\begin{align*}
\Psi_{+}(\mathbb{A} + b) &\cong \{[P , d] \in \mathbb{R}^{n \times (n+1)} : P^\tp  = P^2 = P,\; Pd = 0,\; \tr(P) = l,\; j(\mathbb{A} +b) \subseteq \im([P,d])\}, \\
\Psi_{-}(\mathbb{B} + c) &\cong \{[P , d] \in \mathbb{R}^{n \times (n+1)} : P^\tp  = P^2 = P,\; Pd = 0,\; \tr(P) = k,\; \im([P,d]) \subseteq j(\mathbb{B} + c ) \}.
\end{align*}
\end{proposition}

\section{Algebraic topology of the affine Grassmannian}\label{sec:top}

We will determine the homotopy groups and (co)homology groups/rings of $\Graff(k,n)$. With this in mind, we begin by proving yet another characterization of $\Graff(k,n)$, namely, it is a vector bundle  --- in fact it is the universal quotient bundle of $\Gr(k,n)$.

Recall that if $S$ is a subbundle of a vector bundle $E$ on a manifold $M$, then $Q$ is called the \emph{quotient bundle} on $M$ of $E$ by $S$ if there is a short exact sequence of vector bundles
\begin{equation}\label{eq:quotient}
0 \to S \to  E \to Q \to 0.
\end{equation}
Recall also that the \emph{tautological bundle} over  $\Gr(k,n)$ is the vector bundle whose fiber over $\mathbb{A} \in \Gr(k,n)$ is simply $\mathbb{A}$ itself. One may view this as a subbundle of the  \emph{trivial vector bundle} $\Gr(k,n)\times \mathbb{R}^n$. If $S$ is the tautological bundle and $E$ is the trivial bundle in \eqref{eq:quotient}, then the quotient bundle $Q$ is called the \emph{universal quotient bundle} of $\Gr(k,n)$ \cite{GH, Milnor}.
\begin{theorem}\label{thm:bundle}
\begin{enumerate}[\upshape (i)]
\item\label{it:vb}  $\Graff(k,n)$  is a rank-$(n-k)$ vector bundle over $\Gr(k,n)$  with bundle projection $\tau : \Graff(k,n)\to \Gr(k,n)$, the deaffine map  in \eqref{eq:deaff}.
\item\label{it:uqb} $\Graff(k,n)$ is the universal quotient bundle of $\Gr(k,n)$,
\begin{equation}\label{eq:uqb}
0\to S \to \Gr(k,n)\times \mathbb{R}^n\to \Graff(k,n)\to 0
\end{equation}
where $S$ is the tautological bundle.
\end{enumerate}
\end{theorem}
\begin{proof}
In affine coordinates, the deaffine map 
$\tau: \Graff(k,n)\to \Gr(k,n)$, $\mathbb{A}+b \mapsto \mathbb{A}$ takes the form 
$\tau ([ a_1,\dots, a_k, b_0 ])= [ a_1,\dots, a_k]$
where $a_i$'s and $b_0$ are chosen as in the proof of Proposition~\ref{prop:smooth}. Notice that the fiber $\tau^{-1}(\mathbb{A})$ for $\mathbb{A} \in \Gr(k,n)$ is simply $\mathbb{R}^n/\mathbb{A}$, a linear subspace of dimension $n-k$. Local trivializations of $\Graff(k,n)$ are obtained from local charts of $\Gr(k,n)$ by construction. Hence $\Graff(k,n)$ is a vector bundle over $\Gr(k,n)$. Moreover we have 
$q:\Gr(k,n)\times \mathbb{R}^n \to \Graff(k,n)$, $(\mathbb{A},b) \mapsto \mathbb{A}+b$.
It is straightforward to check that $q$ is a surjective bundle map and the kernel of $q$ is the tautological vector bundle $S$ over $\Gr(k,n)$, i.e., we obtain the exact sequence in \eqref{eq:uqb}.
This shows that $\Graff(k,n)$ is the universal quotient bundle.
\end{proof}
Throughout this section, we write $\mathbb{Z}_2 \coloneqq \mathbb{Z}/2\mathbb{Z}$.

\subsection{Homotopy of $\Graff(k,n)$}

When $\Graff(k,n)$ is regarded as a vector bundle on $\Gr(k,n)$ as in Theorem~\ref{thm:bundle}\eqref{it:vb}, the base space $\Gr(k,n)$ is homeomorphic to the zero section, which is a strong deformation retract of $\Graff(k,n)$. Hence $\Gr(k,n)$ and $\Graff(k,n)$ have the same homotopy type and so 
\[
\pi_r(\Graff(k,n)) \simeq \pi_r(\Gr(k,n)),\quad r\in \mathbb{N}.
\]
From the list of homotopy groups of $\Gr(k,n)$  in \cite[Section 10.8]{FV2013}, we obtain those of $\Graff(k,n)$. 
\begin{proposition}\label{prop:homotopy groups}
$\Graff(k,n)$ is homotopy equivalent to $\Gr(k,n)$. Therefore 
\begin{enumerate}[\upshape (i)]
\item for $n \ge k+ 2$ and $0 < k < n/2$,
\[
\pi_1(\Graff(k,n)) \simeq
\begin{cases}
\mathbb{Z} & \text{if } k =1, n =2,\\
\mathbb{Z}_2 & \text{otherwise};
\end{cases}
\]

\item for $0 \le  k < n/2$ and $2 \le r < n- 2k$,
\[
\pi_r(\Graff(k,n)) \simeq 
\begin{cases}
\mathbb{Z}&\text{if } r=0,4 \bmod 8,\\
\mathbb{Z}_2&\text{if } r=1,2 \bmod 8,\\
0&\text{if } r=3,5,6,7 \bmod 8.
\end{cases}
\]
\end{enumerate}
\end{proposition}

Since the deaffine map $\tau$ in \eqref{eq:deaff} is a bundle projection by Theorem~\ref{thm:bundle}\eqref{it:vb}, it is straightforward to take direct limits in \eqref{eq:comm} and extend Proposition~\ref{prop:homotopy groups} to the infinite Grassmannian via the commutative diagram \eqref{diagram:infinite Stiefel}. This also shows that $\Graff(k,\infty)$ is a \emph{classifying space} \cite{Husemoller1994}.
\begin{corollary}\label{cor:homoinfty}
$\Graff(k,\infty)$ is homotopy equivalent to $\Gr(k,\infty)$. Therefore
\[
\pi_1(\Graff(k,\infty)) \simeq \mathbb{Z}_2;
\]
and for $r \ge 2$,
\[
\pi_r (\Graff(k,\infty)) \simeq 
\begin{cases}
\mathbb{Z}&\text{if } r=0,4 \bmod 8,\\
\mathbb{Z}_2&\text{if } r=1,2 \bmod 8,\\
0&\text{if } r=3,5,6,7 \bmod 8.
\end{cases}
\]
Moreover, $\Graff(k,\infty)$ is the classifying space of $\O(n)$ and $\GL(n)$ with total space $\Vaff(k,\infty)$.
\end{corollary}

\subsection{Homology and cohomology of $\Graff(k,n)$}\label{sec:coho}

We show that the affine Schubert varieties in Section~\ref{sec:schubert} play a role for the  (co)homology of $\Graff(k,n)$ similar to that of Schubert varieties for $\Gr(k,n)$.

Let $\mathbb{A}_1 \subseteq \dots \subseteq \mathbb{A}_k$ be a flag in $\mathbb{R}^n$. For any $b \in \mathbb{R}^n$, the deaffine map $\tau : \Graff(k,n) \to \Gr(k,n)$ in \eqref{eq:deaff}, when restricted to $\Psi(\mathbb{A}_1+b,\dots, \mathbb{A}_k + b)$, defines a map
\[
\tau_b :\Psi(\mathbb{A}_1 + b,\dots, \mathbb{A}_k+ b)  \to  \Omega(\mathbb{A}_1,\dots,\mathbb{A}_k), \quad \mathbb{A} + b \mapsto \mathbb{A}.
\]
For any fixed $b \in \mathbb{R}^n$, it has a right inverse
\[
s_b :\Omega(\mathbb{A}_1,\dots,\mathbb{A}_k) \to \Psi(\mathbb{A}_1 + b,\dots, \mathbb{A}_k+ b), \quad \mathbb{A} \mapsto \mathbb{A} + b.
\]
\begin{lemma}\label{lemma:Schubert1}
Let $\mathbb{A}_1 \subseteq \dots \subseteq \mathbb{A}_k$ be a flag in $\mathbb{R}^n$ and $b \in \mathbb{R}^n$. Then the following diagram commutes:
\begin{equation}\label{diag:schubert}
     \begin{tikzcd}
   \Psi(\mathbb{A}_1 + b,\dots, \mathbb{A}_k+ b) \arrow[hookrightarrow]{r}  \arrow{dr}[anchor=center,xshift=0.5ex,yshift = 1ex]{\tau_b}  &  \tau^{-1}\bigl(\Omega(\mathbb{A}_1,\dots,\mathbb{A}_k)\bigr) \arrow[hookrightarrow]{r}    \arrow{d}{\tau} & \Graff(k,n)   \arrow{d}{\tau}  \\
   &  \Omega(\mathbb{A}_1,\dots,\mathbb{A}_k) \arrow[bend left=10]{ul}[anchor=center,xshift=-2ex]{s_b} \arrow[hookrightarrow]{r}  &  \Gr(k,n)  
     \end{tikzcd}
\end{equation}
\end{lemma}

\begin{proof}
The only point in \eqref{diag:schubert} that needs verification is the inclusion $\Psi(\mathbb{A}_1+b,\dots, \mathbb{A}_k + b) \subseteq \tau^{-1}\bigl(\Omega(\mathbb{A}_1,\dots, \mathbb{A}_k)\bigr)$.  Let $\mathbb{B}+c \in \Psi(\mathbb{A}_1+b,\dots, \mathbb{A}_k + b)$. Then $\dim  (\mathbb{A}_j + b) \cap (\mathbb{B} + c) \ge j$, $ j=1,\dots, k$. We need to show that $\dim \mathbb{A}_j \cap \mathbb{B}\ge j$, $j=1,\dots, k$. By the same argument that led to \eqref{eq:x}, we may choose an $x\in \mathbb{R}^n$ so that
\begin{gather*}
\mathbb{A}_j + b = \mathbb{A}_j + x,\quad \mathbb{B} + c = \mathbb{B} + x,\qquad j=1,\dots,k.
\shortintertext{Therefore}
\dim (\mathbb{A}_j + x) \cap (\mathbb{B} + x) =\dim (\mathbb{A}_j + b) \cap (\mathbb{B} + c) \ge j
\end{gather*}
and thus $\dim \mathbb{A}_j \cap \mathbb{B} \ge j$, $j=1,\dots, k$.
\end{proof}

To obtain a more precise relation between $\Psi(\mathbb{A}_1 + b,\dots, \mathbb{A}_k + b)$ and $\tau^{-1}\bigl(\Omega(\mathbb{A}_1,\dots, \mathbb{A}_k)\bigr)$, we show that the fibers of $\tau_b$ are contractible.
\begin{lemma}\label{lem:Schubert2}
Let $\mathbb{A}_1 \subseteq \dots \subseteq \mathbb{A}_k$ be a flag in $\mathbb{R}^n$, $\mathbb{B}\in \Omega(\mathbb{A}_1,\dots, \mathbb{A}_k)$, and $b \in \mathbb{R}^n$. Then
\[
\tau_b^{-1}(\mathbb{B}) = \{ \mathbb{B} + c\in \Graff(k,n): \dim (\mathbb{B} + c) \cap (\mathbb{A}_j + b) \ge j,\; j=1,\dots, k \}
\]
is convex  and therefore contractible.
\end{lemma}

\begin{proof}
We first define an auxiliary set
\begin{align}
C(\mathbb{B},b)&\coloneqq \{ c \in \mathbb{R}^n : \mathbb{B} + c \in \tau_b^{-1}(\mathbb{B})\} \nonumber\\
&= \{ c \in \mathbb{R}^n :  \dim (\mathbb{B} + c) \cap (\mathbb{A}_j + b) \ge j,\; j=1,\dots,k \} \label{eq:CBb}
\end{align}
If $c\in C(\mathbb{B},b)$, then $c + b'\in C(\mathbb{B},b)$ for any $b'\in \mathbb{B}$. Moreover, $\mathbb{B} + c = \mathbb{B} + c'$ if and only if $c' - c \in \mathbb{B}$. So we have a homeomorphism
\begin{equation}\label{eqn:Schubert1}
C(\mathbb{B},b)/\mathbb{B}  \cong \tau_{b}^{-1}(\mathbb{B}),
\end{equation}
where the left-hand side is regarded as a subset of the quotient vector space $\mathbb{R}^n/\mathbb{B}$. So the convexity of $\tau_{b}^{-1}(\mathbb{B})$ would follow from the convexity of $C(\mathbb{B},b)$ in $\mathbb{R}^n$.

We  remind the reader that if $A \in \V(k,n)$ is an orthonormal basis for $\mathbb{A} \in \Gr(k,n)$, then
$\mathbb{A} = \im(A) = \ker(I - AA^\tp)$. Let $B \in \V(n,n-k)$ and $A_j \in \V(n - d_j, n)$ be orthonormal bases of $\mathbb{B}$ and $\mathbb{A}_j$ respectively, $j=1,\dots,k$. So
\begin{gather}
\mathbb{B} = \{ y \in \mathbb{R}^n : (I - BB^\tp )y = 0 \}, \quad \mathbb{A}_j = \{ y \in \mathbb{R}^n : (I - A_j A_j^\tp) y =0\}, \nonumber
\shortintertext{and so}
\mathbb{B} + c = \{ y \in \mathbb{R}^n : (I - BB^\tp )(y-c) = 0\}, \quad
\mathbb{A}_j + b = \{ y \in \mathbb{R}^n : (I - A_j A_j^\tp) (y-b) =0\},\label{eq:intersect}
\end{gather}
for $j =1,\dots,k$. Hence by \eqref{eq:CBb} and \eqref{eq:intersect},
\[
C(\mathbb{B},b) = \{ c \in \mathbb{R}^n : \text{solution space of \eqref{eq:intersect} has dimension} \ge j,\; j=1,\dots,k\}.
\]
With this characterization of $C(\mathbb{B},b)$, convexity is straightforward: Let $c_1,c_2\in C(\mathbb{B},b)$ and $y_1,y_2 \in \mathbb{R}^n$ be such that 
\begin{equation}\label{eq:als}
(I - BB^\tp)(y_i - c_i)=0,\quad (I - A_jA_j^\tp)(y_i - b) =0,\quad j=1,\dots,k, \quad  i=1,2.
\end{equation}
For any $t\in [0,1]$,  $c_t =t c_1 + (1-t) c_2$ and $y_t = ty_1 + (1-t)y_2$ clearly also satisfy \eqref{eq:als}.
\end{proof}

This leads us to the following relation between $\Psi(\mathbb{A}_1 + b,\dots, \mathbb{A}_k + b)$ and $\tau^{-1}\bigl(\Omega(\mathbb{A}_1,\dots, \mathbb{A}_k)\bigr)$.
\begin{theorem}\label{thm:Schubert2}
The image $s_b\bigl(\Omega(\mathbb{A}_1,\dots,\mathbb{A}_k)\bigr)$ is a strong deformation retract of $\Psi(\mathbb{A}_1 +b, \dots,\mathbb{A}_k + b)$. In particular, $\Psi(\mathbb{A}_1 + b,\dots, \mathbb{A}_k + b)$ is homotopy equivalent to $\tau^{-1}\bigl(\Omega(\mathbb{A}_1,\dots, \mathbb{A}_k)\bigr)$.
\end{theorem}
\begin{proof}
$s_b:\Omega(\mathbb{A}_1,\dots, \mathbb{A}_k) \to \Psi(\mathbb{A}_1+b,\dots, \mathbb{A}_k + b)$ is a section of $\tau_b$, i.e., $\tau_b \circ s_b = 1$.  Hence it suffices to prove that the fiber $\tau_b^{-1}(\mathbb{B})$ deformation retracts to $s_b(\mathbb{B})$ for each $\mathbb{B}\in \Omega(\mathbb{A}_1,\dots, \mathbb{A}_k)$ but this is trivially true since $\tau_b^{-1}(\mathbb{B})$ is contractible by Lemma~\ref{lem:Schubert2}
\end{proof}

By virtue of Theorem~\ref{thm:Schubert2}, we deduce next that the affine Schubert varieties form a natural basis for the (co)homology groups of $\Graff(k,n)$. In this context, the (co)homology classes determined by affine Schubert varieties are called \emph{affine Schubert cycles}.
\begin{theorem}\label{thm:homology} 
Affine Schubert cycles form a basis for the (co)homology groups of an affine Grassmannian. In particular, 
\[
\H_i (\Graff(k,n),\mathbb{Z}_2)  \simeq \H^i(\Graff(k,n),\mathbb{Z}_2) \simeq \mathbb{Z}_2^{r_i},
\]
where $r_i$ is the number of partitions of the integer $i$ with at most $k$ parts.  We also have a graded ring isomorphism
\[
\H^\ast (\Graff(k,\infty),\mathbb{Z}_2) \simeq \mathbb{Z}_2 [x_1,\dots, x_k]^{\mathfrak{S}_k}.
\]
\end{theorem} 
\begin{proof}
By Fact~\ref{fact:schubert}\eqref{it:cyc}, since $\tau$ is a homotopy equivalence, the collection of $\tau^{-1}\bigl(\Omega(\mathbb{A}_1,\dots,\mathbb{A}_k)\bigr)$ over all $i$-dimensional Schubert varieties $\Omega(\mathbb{A}_1,\dots,\mathbb{A}_k)$ form a basis for the $i$th (co)homology group of $\Graff(k,n)$. Therefore, by Theorems~\ref{thm:dim} and \ref{thm:Schubert2}, the $j$-dimensional affine Schubert varieties $\Psi(\mathbb{A}_1+b,\dots,\mathbb{A}_k+b)$ form a basis for the $(j-d_1 +1)$th (co)homology group of $\Graff(k,n)$.

For the cohomology ring, the homotopy equivalence between $\Graff(k,\infty)$ and $\Gr(k,\infty)$ in Corollary~\ref{cor:homoinfty} gives $\H^\ast (\Graff(k,\infty),\mathbb{Z}_2) \simeq \H^\ast (\Gr(k,\infty),\mathbb{Z}_2)$. That $\H^i(\Graff(k,n),\mathbb{Z}_2) \simeq \mathbb{Z}_2^{r_i}$ is a standard result \cite{Prasolov2007, Milnor, Chern1948}. 
\end{proof}
We stated Theorem~\ref{thm:homology} with $\mathbb{Z}_2$ coefficients for simplicity but in the same manner we may obtain $\H^\ast (\Graff(k,\infty),\mathbb{Z})$ and $\H^\ast (\Graff(k,\infty),\mathbb{Q})$  in terms of characteristic classes using the corresponding results for $\Gr(k,\infty)$ in \cite{Borel1953, Takeuchi1962, Brown1982} and \cite{Thomas1960} respectively.

We conclude this section with a cell decomposition of $\Graff(k,n)$, which is not given by affine Schubert varieties but by preimages of Schubert varieties.
\begin{proposition}\label{prop:cell}
The collection of preimages $\tau^{-1}\bigl(\Omega(\mathbb{A}_1,\dots,\mathbb{A}_k)\bigr)$ over all flags in $\mathbb{R}^n$ of length $k$ gives a cell decomposition of $\Graff(k,n)$.
\end{proposition}
\begin{proof}
By Theorem~\ref{thm:bundle}, $\Graff(k,n)$ is a vector bundle over $\Gr(k,n)$ with $\tau : \Graff(k,n) \to \Gr(k,n)$ the bundle projection. By Fact~\ref{fact:schubert}\eqref{it:cell}, the collection of $\Omega(\mathbb{A}_1,\dots,\mathbb{A}_k)$ over all flags of length $k$ provide a cell-decomposition of $\Gr(k,n)$. So the required result follows.
\end{proof}

\section{Metric geometry of the affine Grassmannian}\label{sec:metric}

We have two goals in this section. The first is to extend various distances defined on Grassmannian to the affine Grassmannian, the results are summarized in Table~\ref{tab:distances} --- these are distances between affine subspaces of the \emph{same} dimension. Following our earlier work in \cite{YL}, our next goal is to further extend these distances in a natural way (using the affine Schubert varieties in Definition~\ref{def:Psi+-}) to affine subspaces of \emph{different} dimensions.

\subsection{Issues in metricizing $\Graff(k,n)$}

A  reason for the widespread applicability of the usual Grassmannian is that one has concrete, explicitly computable expressions for geodesics and distances on $\Gr(k,n)$. In  \cite{AMS, EAS, Wong}, these expressions were obtained from a purely differential geometric perspective. One might imagine that the differential geometric structures on $\Graff(k,n)$ in Propositions~\ref{prop:smooth}, \ref{prop:homo}, or Theorem \ref{thm:bundle} would yield similar results. Surprisingly this is not the case.

A more careful examination of  the arguments  in \cite{AMS, EAS, Wong} for obtaining explicit expressions for geodesics and geodesic distances on $\V(k,n)$ and $\Gr(k,n)$ reveal that they rely on a somewhat obscure structure, namely, that of a \emph{geodesic orbit space} \cite{AA, Gordon}. In general, if $G$ is a compact semisimple Lie group and $G/H$ is a reductive homogeneous space, then there is a standard metric induced by the restriction of the  Killing form on $\mathfrak{g}/\mathfrak{h}$ where $\mathfrak{g}$ and $\mathfrak{h}$ are the Lie algebras of $G$ and $H$ respectively. With this standard metric, $G/H$ is a geodesic orbit space, i.e., all geodesics are orbits of one-parameter subgroups of $G$. In the case of $\Gr(k,n)=\O(n)/\bigl(\O(n-k)\times \O(k)\bigr)$ and $\V(k,n)=\O(n)/\O(n-k)$, as $\O(n)$ is a compact semisimple Lie group, $\Gr(k,n)$ and $\V(k,n)$  are  geodesic orbit spaces. Furthermore, as  $\O(n)$ is a matrix Lie group, all its one-parameter subgroups are given by exponential maps, which in turn allows us to  write down explicit expressions for the geodesics (and thus also the geodesic distances) on $\Gr(k,n)$ and $\V(k,n)$.  The difficulty is seeking similar expressions on $\Graff(k,n)=\E(n)/\bigl(\E(n-k)\times \O(k)\bigr)$ is that it may not be a geodesic orbit space since $\E(n)$ is not compact.

What about the vector bundle structure on $\Graff(k,n)$ then? If $E$ is a vector bundle over a Riemannian manifold $M$, then the \emph{pullback} of the metric on $M$ induces a metric on $E$. Nevertheless, this metric on $E$ is uninteresting --- by definition, it disregards the fibers of the bundle. In the context of Theorem~\ref{thm:bundle}, this is akin to defining the distance between $\mathbb{A} + b$ and $\mathbb{B} + c \in \Graff(k,n)$ as the usual Grassmann distance between $\mathbb{A}$ and $\mathbb{B} \in \Gr(k,n)$.

We will turn to the algebraic geometric properties of $\Graff(k,n)$ in Theorem~\ref{thm:alg} to provide the framework for defining distances with explicitly computable expressions, first for equidimensional affine subspaces and next for inequidimensional affine subspaces.

\subsection{Distances on $\Graff(k,n)$}

The Riemannian metric on $\Gr(k,n)$ yields the following well-known \emph{Grassmann distance} between two subspaces $\mathbb{A}, \mathbb{B} \in \Gr(k,n)$,
\begin{equation}\label{eq:grassdist}
d_{\Gr(k,n)}(\mathbb{A},\mathbb{B})=\Bigl(\sum\nolimits_{i=1}^k \theta_i^2\Bigr)^{1/2},
\end{equation}
where  $\theta_1,\dots,\theta_{k}$ are the principal angles between  $\mathbb{A}$ and  $\mathbb{B} $. This distance is easily computable via \textsc{svd} as $\theta_i = \cos^{-1} \sigma_i$, where $\sigma_i$ is the $i$th singular value of  the matrix $A^\tp B$ for any orthonormal bases $A$ and $B$ of $\mathbb{A}$ and $\mathbb{B}$ \cite{GVL, YL}.

By Theorem~\ref{thm:alg}(ii), we may identify $\Graff(k,n)$ with its image $j\bigl(\Graff(k,n)\bigr)$ in $\Gr(k+1,n+1)$. As a subset of $\Gr(k+1,n+1)$, $\Graff(k,n)$ inherits the Grassmann distance $d_{\Gr(k+1,n+1)}$ on $\Gr(k+1,n+1)$, giving us the distance in Theorem~\ref{thm3} that can also be readily computed using \textsc{svd}. We will show in Proposition~\ref{prop:validation} that this distance is in fact \emph{intrinsic}.
\begin{theorem}\label{thm3}
For any two affine $k$-flats $\mathbb{A}+b$ and $\mathbb{B}+c \in \Graff(k,n)$,
\[
d_{\Graff(k,n)}(\mathbb{A}+b,\mathbb{B}+c) \coloneqq d_{\Gr(k+1,n+1)}\bigl(j(\mathbb{A}+b), j(\mathbb{B}+c) \bigr),
\]
where $j$ is the embedding in \eqref{eq:j}, defines a notion of distance consistent with the Grassmann distance.  If
\[
Y_{\mathbb{A} + b} =
\begin{bmatrix}
A& b_0/\sqrt{1+\lVert b_0\rVert^2}\\
0& 1/\sqrt{1+\lVert b_0\rVert^2}
\end{bmatrix},\qquad
Y_{\mathbb{B} + c}=
\begin{bmatrix}
B& c_0/\sqrt{1+\lVert c_0\rVert^2}\\
0& 1/\sqrt{1+\lVert c_0\rVert^2}
\end{bmatrix}
\]
are the matrices of Stiefel coordinates for $\mathbb{A}+b$ and $\mathbb{B}+c$ respectively, then
\begin{equation}\label{eq:graffdist}
d_{\Graff(k,n)}(\mathbb{A}+b,\mathbb{B}+c) = \Bigl(\sum\nolimits_{i=1}^{k+1} \theta_i^2\Bigr)^{1/2},
\end{equation}
where $\theta_i = \cos^{-1} \sigma_i$ and $\sigma_i$ is the $i$th singular value of $Y_{\mathbb{A} + b}^\tp  Y_{\mathbb{B} + c} \in \mathbb{R}^{(k+1) \times (k+1)}$.
\end{theorem}
\begin{proof}
Similar to \cite[Theorem~4.2]{WYL}
\end{proof}

It is not difficult to see that the angles $\theta_1,\dots,\theta_{k+1}$ are independent of the choice of Stiefel coordinates. We define the following affine analogues of principal angles and principal vectors of linear subspaces \cite{BG, GVL, YL} that will be useful later.
\begin{definition}
We will call $\theta_i $ the $i$th \emph{affine principal angles} between the respective affine subspaces and denote it by $\theta_i(\mathbb{A}+b,\mathbb{B}+c)$. Consider the \textsc{svd},
\begin{equation}\label{eq:scsvd}
 Y_{\mathbb{A} + b}^\tp  Y_{\mathbb{B} + c} = U \Sigma V^\tp 
\end{equation}
where $U,V \in \O(k+1)$ and $\Sigma = \diag (\sigma_1,\dots,\sigma_{k+1})$. Let
\[
 Y_{\mathbb{A} + b} U = [p_1,\dots,p_{k+1}], \qquad Y_{\mathbb{B} + c} V = [q_1,\dots,q_{k+1}].
\]
We will call the pair of column vectors $(p_i,q_i)$ the $i$th \emph{affine principal vectors} between $\mathbb{A}+b$ and $\mathbb{B}+c$.
\end{definition}
We next show that the distance in Theorem~\ref{thm3} is the only possible distance on an affine Grassmannian compatible with the usual Grassmann distance on a Grassmannian. On any connected Riemannian manifold $M$ with Riemannian metric $g$, there is an intrinsic distance function $d_M$ on $M$ with respect to $g$,
\[
d_M(x,y) \coloneqq \inf \{ L(\gamma):  \text{$\gamma$ is a piecewise smooth curve connecting $x$ and $y$ in $M$}\}.
\]
Here $L(\gamma)$ is the length of the cruve $\gamma : [0,1] \to M$ defined by 
\[
L(\gamma) \coloneqq \int_0^1 \lVert \gamma'(t) \rVert = \int_0^1 \sqrt{\smash[b]{g_{\gamma(t)}}(\smash[t]{\gamma'(t),\gamma'(t)})}.
\]

For a connected submanifold of  $N \subseteq M$, there is a natural Riemannian metric $g_N$ on $N$ induced by $g$ and therefore a corresponding intrinsic distance function,
\[
d_N(x,y) \coloneqq \inf \{ L(\gamma):  \text{$\gamma$ is a piecewise smooth curve connecting $x$ and $y$ in $N$}\}.
\]
On the other hand, we may also define a distance function $d_M\rvert_N$ on $N$ by simply restricting the distance function $d_M$ to $N$ --- note that this is what we have done in Theorem~\ref{thm3} with $M=\Gr(k+1,n+1)$ and $N=\Graff(k,n)$. In general,
$d_M\rvert_N \ne d_N$.
For example, for $N = \mathbb{S}^2$ embedded as the unit sphere in $M =\mathbb{R}^3$, the two distance functions on $\mathbb{S}^2$  are obviously different. However, for our embedding of $\Graff(k,n)$ in $\Gr(k+1,n+1)$, the two distances on $\Graff(k,n)$ agree. 
\begin{proposition}\label{prop:general validation}
Let $K$ be a closed submanifold of codimension at least two in $M$ and let $N$ be the complement of $K$ in $M$. Then
$d_M\rvert_N = d_N$.
\end{proposition} 
\begin{proof}
We need to show that for any two distinct points $x,y\in N$,
$d_M(x,y)=d_N(x,y)$.
By definition of $d_M$ and $d_N$ it suffices to show that any piecewise smooth curve $\gamma$ in $M$ connecting $x,y$ can be approximated by a piecewise smooth curve in $N$ connecting $x,y$.
The assumption on codimension implies that $x,y\in N$ is connected by a piecewise smooth curve in $N$. The transversality theorem \cite[Theorem~2.4]{Hirsch} then implies that  $\gamma$ can be approximated by curves in $N$ connecting $x$ and $y$.
\end{proof}
\begin{proposition}\label{prop:validation}
The distance $d_{\Graff(k,n)}$ in Theorem~\ref{thm3} is intrinsic with respect to the Riemannian metric on $\Graff(k,n)$ induced from that of $\Gr(k+1,n+1)$.
\end{proposition}
\begin{proof}
By Theorem~\ref{thm:alg}, the complement of $N = \Graff(k,n)$ in $M$ is $\Gr(k+1,n)$ and has codimension $k+1\ge 2$. Hence Proposition~\ref{prop:general validation} applies.
\end{proof}

At this point, we believe we have provided sufficient justification to call the distance in \eqref{eq:graffdist} the \emph{Grassmann distance} on $\Graff(k,n)$.
We next determine an expression for the geodesic connecting two points on $\Graff(k,n)$ that attains their minimum Grassmann distance. There is one caveat --- this geodesic may contain a point lying outside  $\Graff(k,n)$ as it is not a geodesically complete manifold.
\begin{lemma}
Let $\mathbb{A}+b,\mathbb{B}+c\in \Graff(k,n)$ and let
\[
Y_{\mathbb{A} + b} = \begin{bmatrix} A & b_0/\sqrt{\lVert b_0\rVert^{2}+1} \\ 0& 1/\sqrt{\lVert b_0\rVert^{2}+1} \end{bmatrix},\qquad Y_{\mathbb{B} + c} =\begin{bmatrix} B& c_0/\sqrt{\lVert c_0\rVert^{2}+1} \\ 0& 1/\sqrt{\lVert c\rVert^{2}+1} \end{bmatrix} 
\]
be their Stiefel coordinates. If $Y_{\mathbb{A} + b}^\tp Y_{\mathbb{B} + c}$ is invertible, then there is at most one point on the distance minimizing geodesic in $\Gr(k+1,n+1)$ connecting $\mathbb{A}+b$ and $\mathbb{B}+c$ which lies outside $j\bigl(\Graff(k,n)\bigr)$. Here $j$ is the embedding in \eqref{eq:j}.
\end{lemma}
\begin{proof}
Let $U \in \O(k+1)$ and the diagonal matrix $\Sigma$ be as in \eqref{eq:scsvd}.  Let  $\Theta \coloneqq \diag(\theta_1,\dots,\theta_{k+1}) = \cos^{-1} \Sigma$ be the diagonal matrix of affine principal angles. By \cite{AMS} the geodesic $\gamma : [0,1] \to \Gr(k+1,n+1)$ connecting $j(\mathbb{A}+b)$ and $j(\mathbb{B}+c)$ is given by
$\gamma(t) = \spn\bigl( Y_{\mathbb{A}+b}U\cos(t\Theta)+Q\sin(t\Theta) \bigr)$,
where $Q \in \O(k+1)$ is such that the \textsc{rhs} of
\[
(I-Y_{\mathbb{A} + b}Y_{\mathbb{A} + b}^\tp )Y_{\mathbb{B} + c}(Y_{\mathbb{A} + b}^\tp Y_{\mathbb{B} + c})^{-1} = Q(\tan(\Theta))U^\tp 
\]
gives an \textsc{svd} of the matrix on the \textsc{lhs}.  Let the last row of $U$ and $Q$ as $[u_{k+1,1},\dots,u_{k+1,k+1}]^\tp $ and $[q_{k+1,1},\dots,q_{k+1,k+1}]^\tp $ respectively. Then $\gamma(t) \in \Gr(k+1,n+1)\setminus j\bigl(\Graff(k,n)\bigr)$ if and only if the entries on last row of $\gamma(t)$ are all zero, i.e.,
\[
\frac{u_{k+1,i}\cos(t\theta_i)}{\sqrt{\lVert b\rVert^{2}+1}}+q_{k+1,i}\sin(t\theta_i) = 0,
\]
for all $ i = 1,\dots ,k+1$. So at most one point on $\gamma$ lies outside $\Graff(k,n)$.
\end{proof}
\begin{corollary} \label{corollary:geodesic}
Let $\mathbb{A}+b$ and $\mathbb{B}+c\in \Graff(k,n)$. The distance minimizing geodesic $\gamma : [0,1]\to \Graff(k,n)$ connecting $\mathbb{A}+b$ and $\mathbb{B}+c$ is given by
\begin{equation}\label{eq:geodesic}
\gamma(t) = j^{-1}\bigl( \spn(Y_{\mathbb{A} + b} U \cos t\Theta + Q\sin t\Theta)\bigr),
\end{equation}
where $Q,U \in \O(k+1)$ and the diagonal matrix $\Theta$ are determined by the \textsc{svd}
\[
(I - Y_{\mathbb{A} + b} Y_{\mathbb{A} + b}^\tp ) Y_{\mathbb{B} + c} (Y_{\mathbb{A} + b}^\tp  Y_{\mathbb{B} + c})^{-1} = Q (\tan \Theta) U^\tp .
\]
The matrix $U$ is the same as that in \eqref{eq:scsvd} and $\Theta = \diag(\theta_1,\dots, \theta_{k+1})$ is the diagonal matrix of affine principal angles.
$\gamma$ attains the distance in \eqref{eq:graffdist} and its derivative at $t=0$ is given by
\begin{equation}\label{eq:log}
\gamma'(0) =  j^{-1}\bigl(Q \Theta U^\tp \bigr).
\end{equation}
\end{corollary}

The Grassmann distance in \eqref{eq:grassdist} is the best known distance on the Grassmannian. But there are in fact several common distances on the Grassmannian \cite[Table~2]{YL} and we may extend them to the affine Grassmannian by applying the embedding $j : \Graff(k,n) \to \Gr(k+1,n+1)$ and emulating our arguments in this section. We summarize these distances in Table~\ref{tab:distances}.
\begin{table*}[ht]
\caption{Distances on $\Graff(k,n)$ in terms of affine principal angles and Stiefel coordinates. The matrices $U, V \in \O(k+1)$ in the right column of Table~\ref{tab:distances} are the ones in \eqref{eq:scsvd}.}
\label{tab:distances}
\centering
\renewcommand{\arraystretch}{1.75}
\small
\begin{tabular}{lll}
 & \textit{Affine principal angles} & \textit{Stiefel coordinates}\\
Asimov & $d^{\alpha}_{\Graff(k,n)}(\mathbb{A} + b, \mathbb{B} + c) =  \theta_{k+1}$ & $\cos^{-1} \lVert Y_{\mathbb{A} + b}^\tp  Y_{\mathbb{B} + c} \rVert_2$\\
Binet--Cauchy & $d^{\beta}_{\Graff(k,n)}(\mathbb{A} + b, \mathbb{B} + c) = \left(1 - \prod_{i=1}^{k+1}\cos^2\theta_i\right)^{1/2}$ & $(1 - (\det Y_{\mathbb{A} + b}^\tp  Y_{\mathbb{B} + c})^2)^{1/2}$\\
Chordal & $d^{\kappa}_{\Graff(k,n)}(\mathbb{A} + b, \mathbb{B} + c) = \left(\sum_{i=1}^{k+1}\sin^2\theta_i\right)^{1/2}$ & $\frac{1}{\sqrt{2}} \lVert Y_{\mathbb{A} + b}Y_{\mathbb{A} + b}^\tp  - Y_{\mathbb{B} + c}Y_{\mathbb{B} + c}^\tp  \rVert_F$\\
Fubini--Study & $d^{\phi}_{\Graff(k,n)}(\mathbb{A} + b, \mathbb{B} + c) = \cos^{-1} \left(\prod_{i=1}^{k+1}\cos \theta_i\right)$ & $\cos^{-1} \lvert \det Y_{\mathbb{A} + b}^\tp  Y_{\mathbb{B} + c} \rvert$\\
Martin & $d^{\mu}_{\Graff(k,n)}(\mathbb{A} + b, \mathbb{B} + c) = \left( \log \prod_{i=1}^{k+1} 1/\cos^2 \theta_i \right)^{1/2}$ & $(-2 \log \det Y_{\mathbb{A} + b}^\tp  Y_{\mathbb{B} + c})^{1/2}$\\
Procrustes & $d^{\rho}_{\Graff(k,n)}(\mathbb{A} + b, \mathbb{B} + c) = 2\left(\sum_{i=1}^{k+1}\sin^2(\theta_i/2)\right)^{1/2}$ & $\lVert Y_{\mathbb{A} + b}U - Y_{\mathbb{B} + c}V \rVert_F$\\
Projection & $d^{\pi}_{\Graff(k,n)}(\mathbb{A} + b, \mathbb{B} + c) = \sin \theta_{k+1}$ & $\lVert Y_{\mathbb{A} + b}Y_{\mathbb{A} + b}^\tp  - Y_{\mathbb{B} + c}Y_{\mathbb{B} + c}^\tp  \rVert_2$\\
Spectral & $d^{\sigma}_{\Graff(k,n)}(\mathbb{A} + b, \mathbb{B} + c) = 2\sin (\theta_{k+1}/2)$ & $\lVert Y_{\mathbb{A} + b}U - Y_{\mathbb{B} + c}V \rVert_2$
\end{tabular}
\end{table*}

\subsection{Distances on $\Graff(\infty,\infty)$}\label{sec:distdifferent}

The problem of defining distances between \emph{linear} subspaces of different dimensions has recently been resolved in \cite{YL}.  We show here that the  framework in \cite{YL} may be adapted  for \emph{affine} subspaces. This is expected to be important in modeling mixtures of affine subspaces of different dimensions \cite{SLLM}. The proofs of Lemma~\ref{lem:ainfty}, Theorems~\ref{thm2} and \ref{thm:aothermetrics} are similar to those of their linear counterparts \cite[Lemma~3, Theorems~7 and 12]{YL}  and are omitted.

Our first observation is that the Grassmann distance \eqref{eq:graffdist} on $\Graff(k,n)$ does not depend on the ambient space $\mathbb{R}^n$ and may thus be extended to $\Graff(k,\infty)$.
\begin{lemma}\label{lem:ainfty}
The value $d_{\Graff(k,n)}(\mathbb{A}+b,\mathbb{B}+c)$ of two $k$-flats $\mathbb{A}+b$ and $\mathbb{B}+c \in \Graff(k,n)$ is independent of $n$, the dimension of their ambient space. Consequently, $d_{\Graff(k,n)}$ induces a distance $d_{\Graff(k,\infty)}$ on $\Graff(k,\infty)$.
\end{lemma}

Our second observation is that for a $k$-dimensional affine subspace $\mathbb{A} + b$ and an $l$-dimensional affine subspace $\mathbb{B} + c$, assuming $k \le l$ without loss of generality, (i) the distance from $\mathbb{A} + b$ to the set of $k$-dimensional affine subspaces contained in $\mathbb{B} + c$ \emph{equals} (ii) the distance from $\mathbb{B} + c$ to the set of $l$-dimensional affine subspaces containing $\mathbb{A} + b$. Their common value then defines a natural distance between  $\mathbb{A} + b$ and $\mathbb{B} + c$.

Note that (i) is a distance in $\Graff(k,n)$ whereas (ii) is a distance in $\Graff(l,n)$. Furthermore, the set in (i) is precisely  $\Psi_{+}(\mathbb{A}+b)$ and the set in (ii) is precisely $\Psi_{-}(\mathbb{B}+c)$ --- the affine Schubert varieties introduced in Definition~\ref{def:Psi+-}.
\begin{theorem}\label{thm2}
Let $k \le l \le n$.
For any  $\mathbb{A}+b \in \Graff(k,n)$ and $\mathbb{B}+c \in \Graff(l,n)$, the following distances are equal,
\begin{equation}\label{eq:equalv}
d_{\Graff(k,n)}\bigl(\mathbb{A}+b,\Psi_{-}(\mathbb{B}+c)\bigr)=d_{\Graff(l,n)}\bigl(\mathbb{B}+c, \Psi_{+}(\mathbb{A}+b)\bigr),
\end{equation}
and their common value $\delta(\mathbb{A}+b,\mathbb{B}+c)$ may be computed explicitly as
\begin{equation}\label{eq:graffdelta}
\delta(\mathbb{A}+b,\mathbb{B}+c)=\Bigl(\sum\nolimits_{i=1}^{\min(k,l)+1}\theta_i(\mathbb{A}+b,\mathbb{B}+c)^2 \Bigr)^{1/2}.
\end{equation}
\end{theorem}
The affine principal angles $\theta_1,\dots,\theta_{\min(h,l)+1}$ are as defined in Theorem~\ref{thm3} except that now they correspond to the singular values of the rectangular matrix 
\[
Y_{\mathbb{A} + b}^\tp  Y_{\mathbb{B} + c} =
\begin{bmatrix}
A& b_0/\sqrt{1+\lVert b_0\rVert^2}\\
0& 1/\sqrt{1+\lVert b_0\rVert^2}
\end{bmatrix}^\tp 
\begin{bmatrix}
B& c_0/\sqrt{1+\lVert c_0\rVert^2}\\
0& 1/\sqrt{1+\lVert c_0\rVert^2}
\end{bmatrix} \in \mathbb{R}^{(k+1) \times (l+1)}.
\]
Like its counterpart for linear subspaces in \cite[Theorem~7]{YL}, $\delta$ defines a distance between the respective affine subspaces in the sense of a distance of a point to a set. It  reduces to the Grassmann distance $d_{\Graff(k,n)}$ in  \eqref{eq:graffdist} when $\dim \mathbb{A} = \dim \mathbb{B} = k$.

Our third observation is that, like $d_{\Graff(k,n)}$, the distances in Table~\ref{tab:distances} may be extended in the same manner to affine subspaces of different dimensions.
\begin{theorem}\label{thm:aothermetrics}
Let $k \le l \le n$. Let $\mathbb{A}+b\in \Graff(k,n)$, $\mathbb{B}+c\in\Graff(l,n)$. Then
\[
 d^*_{\Graff(k,n)}\bigl(\mathbb{A}+b, \Psi_{-}(\mathbb{B}+c)\bigr) =  d^*_{\Graff(l,n)}\bigl(\mathbb{B}+c, \Psi_{+}(\mathbb{A}+b)\bigr)
\]
for  $* = \alpha, \beta, \kappa, \mu, \pi, \rho, \sigma, \phi$. Their common value $\delta^*(\mathbb{A}+b,\mathbb{B}+c)$ is given by:
{\small
\begin{align*}
\delta^{\alpha}(\mathbb{A}+b, \mathbb{B}+c) &=  \theta_{k+1}, &\delta^{\beta}(\mathbb{A}+b, \mathbb{B}+c) &= \Bigl(1 - \prod\nolimits_{i=1}^{k+1}\cos^2\theta_i\Bigr)^{1/2},\\
\delta^{\pi}(\mathbb{A}+b, \mathbb{B}+c) &= \sin \theta_{k+1}, &\delta^{\mu}(\mathbb{A}+b, \mathbb{B}+c) &= \Bigl( \log \prod\nolimits_{i=1}^{k+1}\frac{1}{\cos^2 \theta_i}\Bigr)^{1/2},\\ 
\delta^{\sigma}(\mathbb{A}+b, \mathbb{B}+c) &= 2\sin (\theta_{k+1}/2),  &\delta^{\phi}(\mathbb{A}+b, \mathbb{B}+c) &= \cos^{-1}\bigl(\prod\nolimits_{i=1}^{k+1}\cos \theta_i\Bigr),\\
\delta^{\kappa}(\mathbb{A}+b, \mathbb{B}+c) &= \Bigl(\sum\nolimits_{i=1}^{k+1}\sin^2\theta_i\Bigr)^{1/2}, &\delta^{\rho}(\mathbb{A}+b, \mathbb{B}+c) &= \Bigl(2\sum\nolimits_{i=1}^{k+1}\sin^2(\theta_i/2)\Bigr)^{1/2},
\end{align*}}
where $\theta_1,\dots,\theta_{k+1}$ are as defined above.
\end{theorem}

Like the $\delta$ in Theorem~\ref{thm2}, the $\delta^*$'s in Theorem~\ref{thm:aothermetrics} are distances in the sense of distances from a point to a set, but they are not \emph{metrics}. The doubly infinite Grassmannian of linear subspaces of all dimensions  $\Gr(\infty,\infty)$ has been shown to be metrizable  \cite[Section~5]{YL} with respect to any of the common distances between linear subspaces. 

Our last observation is that $\Graff(\infty,\infty)$ can likewise be metricized, i.e., a metric can be defined between any pair of affine subspaces of arbitrary dimensions. The embedding $j:\Graff(k,n)\to \Gr(k+1,n+1)$ induces an embedding of sets
$j_\infty:\Graff(\infty,\infty)\to \Gr(\infty,\infty)$.
So $\Graff(\infty,\infty)$ may be identified with $j_\infty\bigl(\Graff(\infty,\infty)\bigr)$ and regarded as a subset of $\Gr(\infty,\infty)$. It inherits any metric on $\Gr(\infty,\infty)$: If $\mathbb{A}+b$ and $\mathbb{B}+c$ are affine subspaces of possibly different dimensions, we may define
\[
d^*_{\Graff(\infty,\infty)}(\mathbb{A}+b,\mathbb{B}+c) \coloneqq d^*_{\Gr(\infty,\infty)}\bigl(j_\infty(\mathbb{A}+b),j_\infty(\mathbb{B}+c)\bigr),
\]
for any choice of metric $d^*_{\Gr(\infty,\infty)}$ on $\Gr(\infty,\infty)$. For example, the metrics  in Table~\ref{tab:metricaff} correspond to Grassmann, chordal, and Procrustes distances.
\begin{table}[ht]
\caption{Metrics on $\Graff(\infty,\infty)$ in terms of affine principal angles and $k = \dim \mathbb{A}$, $l = \dim \mathbb{B}$.}
\centering
\label{tab:metricaff}
\renewcommand{\arraystretch}{1.75}
\small
\begin{tabular}{ll}
Grassmann metric & $d_{\Graff(\infty, \infty)}(\mathbb{A}+b, \mathbb{B}+c) = \Bigl(\lvert k - l \rvert\pi^2/4 + \sum_{i=1}^{\min(k+1,l+1)}\theta_i^2\Bigr)^{1/2}$\\
Chordal metric & $d_{\Graff(\infty, \infty)}^{\kappa}(\mathbb{A}+b, \mathbb{B}+c) = \Bigl(\lvert k - l \rvert + \sum_{i=1}^{\min(k+1,l+1)} \sin^2\theta_i\Bigr)^{1/2}$ \\
Procrustes metric & $d_{\Graff(\infty, \infty)}^{\rho}(\mathbb{A}+b, \mathbb{B}+c) =\Bigl(\lvert k - l \rvert + 2\sum_{i=1}^{\min(k+1,l+1)}\sin ^2(\theta_i/2)\Bigr)^{1/2}$
\end{tabular}
\end{table}

\section{Probability on the affine Grassmannian}\label{sec:prob}

To do statistical estimation and inference with affine subspace-valued data, i.e., with $\Graff(k,n)$ in place of $\mathbb{R}^n = \Graff(0,n)$, we will need reasonable notions of probability densities on $\Graff(k,n)$. We introduce three here: uniform, Langevin (or von Mises--Fisher), and Langevin--Gaussian.

The Riemannian metric on $\Gr(k,n)$ that induces the Grassmann distance in \eqref{eq:grassdist} also induces a volume density $d\gamma_{k,n}$ on $\Gr(k,n)$ \cite[Proposition~9.1.12]{Nicolaescu} with
\begin{equation}\label{eq:volume}
\Vol\bigl(\Gr(k,n)\bigr)=\int_{\Gr(k,n)} |d\gamma_{k,n}| =\binom{n}{k}\frac{\prod_{j=1}^{n}{\omega_j}}{\bigl(\prod_{j=1}^k\omega_j \bigr)\bigl(\prod_{j=1}^{n-k}\omega_j\bigr)},
\end{equation}
where $\omega_m \coloneqq \pi^{m/2}/\Gamma(1+m/2)$, volume of the unit ball in $\mathbb{R}^m$. A natural uniform probability density on $\Gr(k,n)$ is given by $d\mu_{k,n} \coloneqq \Vol\bigl(\Gr(k,n)\bigr)^{-1} \lvert d\gamma_{k,n}\rvert$.

By Theorem~\ref{thm:alg}(ii), $\Graff(k,n)$ is a Zariski open dense subset in $\Gr(k+1,n+1)$ and we must have
$\mu_{k+1,n+1}\bigl(\Graff(k,n)\bigr) = 1$.
Therefore the restriction of $\mu_{k+1,n+1}$ to $\Graff(k,n)$ gives us a \emph{uniform probability measure} on $\Graff(k,n)$. It has an interesting property --- a volumetric analogue of Theorem~\ref{thm2}: The probability that a randomly chosen $l$-dimensional affine subspace  contains  $\mathbb{A} + b$ equals the probability that a randomly chosen $k$-dimensional affine subspace is contained in $\mathbb{B} + c$.
\begin{theorem}
Let $k \le l \le n$ be such that $k+l \ge n$. Let $\mathbb{A}+b\in \Graff(k,n)$ and $\mathbb{B}+c\in\Graff(l,n)$. The relative volume of $\Psi_{+}(\mathbb{A}+b)$ in $\Graff(l,n)$ and $\Psi_{-}(\mathbb{B}+c)$ in $\Graff(k,n)$ are identical.
Furthermore, their common value does not depend on the choices of $\mathbb{A}+b$ and $\mathbb{B}+c$ but only on $k,l,n$ and is given by
\[
\mu_{l+1,n+1}\bigl(\Psi_{+}(\mathbb{A}+b)\bigr)=\mu_{k+1,n+1}\bigl(\Psi_{-}(\mathbb{B}+c)\bigr) = \frac{(l+1)!(n-k)!\prod_{j=l-k+1}^{l+1}\omega_j}{(n+1)!(l-k)!\prod_{j=n-k+1}^{n+1}\omega_j}.
\]
\end{theorem}
\begin{proof}
By Theorem~\ref{thm:alg}(ii), we have
\[
\Vol\bigl(\Graff(k,n)\bigr)=\Vol\bigl(\Gr(k+1,n+1)\bigr)=\binom{n+1}{k+1}\frac{\prod_{j=1}^{n+1}{\omega_j}}{\bigl(\prod_{j=1}^{k+1}\omega_j\bigr)\bigl(\prod_{j=1}^{n-k}\omega_j)}.
\]
By Proposition~\ref{prop:OmegaGraff}, we have 
\begin{align*}
\Vol \bigl(\Psi_{+}(\mathbb{A}+b)\bigr)&=\Vol \bigl(\Gr(n-l,n-k)\bigr)=
\binom{n-k}{n-l}\frac{\prod_{j=1}^{n-k}{\omega_j}}{\bigl(\prod_{j=1}^{n-l}\omega_j\bigr)\bigl(\prod_{j=1}^{l-k}\omega_j\bigr)},\\
\Vol\bigl(\Psi_{-}(\mathbb{B}+c)\bigr)&=\Vol \bigl(\Graff(k,l)\bigr)=
\binom{l+1}{k+1}\frac{\prod_{j=1}^{l+1}{\omega_j}}{\bigl(\prod_{j=1}^{k+1}\omega_j\bigr)\bigl(\prod_{j=1}^{l-k}\omega_j\bigr)}.
\end{align*}
Dividing $\Vol \bigl(\Psi_{+}(\mathbb{A}+b)\bigr)$ and $\Vol\bigl(\Psi_{-}(\mathbb{B}+c)\bigr)$ by $\Vol \bigl(\Graff(l,n)\bigr)$ and $\Vol \bigl(\Graff(k,n)\bigr)$ respectively  completes the proof.
\end{proof}

In the following we will use  the projection coordinates in Definition~\ref{def:proj}. By embedding $\Graff(k,n)$ as a subset $X = j\bigl(\Graff(k,n)\bigr) \subseteq \Gr(k+1,n+1)$ as in Theorem~\ref{thm:alg}(ii) and noting that $X$ is an open dense subset, we have  $\mu(X) = 1$ for any Borel probability measure $\mu$ on $\Gr(k+1,n+1)$ (and that $\mu(X^c) = 0$). Hence $\Graff(k,n)$ inherits any continuous probability distribution on $\Gr(k+1,n+1)$, in particular the Langevin distribution \cite{CY}.
\begin{definition}
The \emph{Langevin distribution}, also known as the \emph{von Mises--Fisher distribution}, on $\Graff(k,n)$ is given by the probability density function
\[
f_L(P_{\mathbb{A} + b} \mid S) \coloneqq \frac{1}{\Hypergeometric{1}{1}{\frac{1}{2}(k+1)}{\frac{1}{2}(n+1)}{S}}\exp \bigl( \tr (SP_{\mathbb{A} + b})\bigr)
\]
for any $\mathbb{A} + b \in \Graff(k,n)$. Here $S \in \mathbb{R}^{(n+1) \times (n+1)}$ is symmetric and
$\prescript{}{1}F_{1}$ is the confluent hypergeometric function of the first kind of a matrix argument \cite{KE}.
\end{definition}
$\Hypergeometric{1}{1}{a}{b}{S} $ has well-known expressions as series and integrals and may be characterized via functional equations and recurrence relations.  However, its explicit expression is unimportant for us ---  the only thing to note is that it can be efficiently evaluated \cite{KE} for any  $a,b\in \mathbb{C}$ and symmetric $S \in \mathbb{C}^{(n+1) \times (n+1)}$.

Roughly speaking, the parameter $S \in \mathbb{R}^{(n+1) \times (n+1)}$ may be interpreted as a `mean direction' and its eigendecomposition $S = V \Lambda V^\tp $ gives an `orientation' $V \in \O(n+1)$ with `concentrations' $\Lambda = \diag(\lambda_1,\dots,\lambda_{n+1})$. In some sense, the Langevin distribution measures the first-order  `spread' on $\Graff(k,n)$. If $S = 0$, then the distribution reduces to the uniform distribution but if $S$ is `large' (i.e., $\lvert \lambda_i\rvert$'s are large), then the distribution concentrate about the orientation $V$. It may appear that a `Bingham distribution'  that measures second-order  `spread' can be defined by
\[
f_B(P_{\mathbb{A} + b} \mid S) \coloneqq \frac{1}{\Hypergeometric{1}{1}{\frac{1}{2}(k+1)}{\frac{1}{2}(n+1)}{S}}\exp \bigl( \tr (P_{\mathbb{A} + b}SP_{\mathbb{A} + b})\bigr)
\]
on $\Graff(k,n)$,
but this is identical to the Langevin distribution as $\tr (PSP) = \tr (SP^2) = \tr(SP)$ for any projection matrix $P$.

The Langevin distribution treats an affine subspace $\mathbb{A} + b \in \Graff(k,n)$ as a single object but there are occasions where it is desirable to distinguish between the linear subspace $\mathbb{A}\in \Gr(k,n)$ and the displacement vector $b \in \mathbb{R}^n$. We will show how a probability distribution on $\Graff(k,n)$ may be constructed by amalgamating probability distributions on $\Gr(k,n)$ and $\mathbb{R}^n$ (or rather, $\mathbb{R}^{n-k}$, as we will see). First, we will identify $\Gr(k,n)$ and $\Graff(k,n)$ with their projection affine coordinates, i.e., imposing equality in \eqref{eq:grproj} and \eqref{eq:graffmatrix},
\begin{align*}
\Gr(k,n) &= \{P \in \mathbb{R}^{n \times n}: P^\tp  = P^{2} = P, \; \tr(P) = k\},\\
\Graff(k,n) &= \{[P , b] \in \mathbb{R}^{n \times (n+1)}: P \in \Gr(k,n), \; Pb = 0\}.
\end{align*}
We will define a marginal density on the linear subspaces, and then impose a conditional density on the  displacement vectors in the orthogonal complement of the respective linear subspaces.

For concreteness, we use the Langevin distribution $f_L(P \mid S)$ on the linear spaces $P \in \Gr(k,n)$. Conditioning on $P$, we know there exists $Q \in \O(n)$ such that
$\ker(P) = \{b \in \mathbb{R}^n : Pb = 0\} = Q E_{n-k} \cong \mathbb{R}^{n-k}$,
where $E_{n-k} \coloneqq \spn\{e_1,\dots,e_{n-k}\} \subseteq \mathbb{R}^{n+1}$. We may use any probability distribution on $\ker(P) \cong \mathbb{R}^{n-k}$ but again for concreteness, a natural choice is the \emph{spherical Gaussian distribution} with probability density $f_G(x\mid \sigma^2) \coloneqq (2\pi\sigma^{2})^{-(n-k)/2} \exp (-\lVert x\rVert^{2}/2 \sigma^{2})$. The conditional density on $\ker(P)$ is then 
\begin{equation}\label{eq:condG}
f_G(b\mid P, \sigma^2) = \frac{1}{\sqrt{(2\pi\sigma^{2})^{n-k}}} \exp\left(-\frac{\lVert b\rVert^{2}}{2 \sigma^{2}}\right)
\end{equation}
for any $b \in \ker(P)$. Note that $Q^\tp b = \begin{bsmallmatrix} b' \\ 0\end{bsmallmatrix}$ where $b' \in \mathbb{R}^{n-k}$ and since $\lVert b\rVert = \lVert Q^\tp b\rVert = \lVert b'\rVert$, it is fine to have $b$ instead of $b'$ appearing on the \textsc{rhs} of \eqref{eq:condG}.  The construction gives us the following distribution.
\begin{definition}
The probability density function of the \emph{Langevin--Gaussian distribution} on $\Graff(k,n)$ is $f_{LG}([P , b]\mid S,\sigma^2) \coloneqq f_L(P\mid S) f_G(b\mid P, \sigma^2)$, i.e.,
\[
f_{LG}([P , b]\mid S,\sigma^2)= \frac{1}{\Hypergeometric{1}{1}{\frac{1}{2}k}{\frac{1}{2}n}{S}\sqrt{(2\pi\sigma^{2})^{n-k}}}\exp\biggl(\tr (SP)-\frac{\lVert b\rVert^{2}}{2 \sigma^{2}}\biggr),
\]
where $S \in \mathbb{R}^{n \times n}$ is symmetric and $\sigma^2> 0$.
\end{definition}

\section{Statistics on the affine Grassmannian}\label{sec:stat}

This section bears little relation to Section~\ref{sec:prob}. Instead of considering statistical analysis of affine subspace-valued data, we argue that the affine Grassmannian is hidden in plain sight in many standard problems of old-fashioned statistics and machine learning.

Statistical estimation problems in multivariate data analysis and machine learning often seek linear relations among variables. This translates to finding an affine subspace from  the sample data set that, in an appropriate sense, either best represents the data set or best separates it into components. In other words, statistical estimation problems are often optimization problems on the affine Grassmannian. We present four examples to illustrate this, following conventional statistical notations ($n$, $p$, $X$, $\beta$, etc).
\begin{example}[Linear Regression]\label{eg:lr}
Consider a linear regression problem with $X \in \mathbb{R}^{n \times p}$, a design matrix of explanatory variables, and $y \in \mathbb{R}^n$, a vector of response variables. Let $\mathbbm{1} = [1,\dots,1]^\tp  \in \mathbb{R}^n$ and $e_{p+1} = [0,\dots,0,1]^\tp  \in \mathbb{R}^{p+1}$. Set $\widetilde{X} = [X, \mathbbm{1} ]\in \mathbb{R}^{n \times (p+1)}$ and define the affine subspace
$\bigl\{\begin{bsmallmatrix} z \\ \beta^\tp z \end{bsmallmatrix} \in \mathbb{R}^{p+1} : z \in \mathbb{R}^p \bigr\}+ \beta_{p+1}e_{p+1}$,
chosen so that  $\widetilde{\beta} = [\beta, \beta_{p+1}]^\tp  \in \mathbb{R}^{n+1}$ minimizes the sum of squared residuals $\lVert \widetilde{X}\widetilde{\beta}  -  y\rVert^2$. Then $\beta \in \mathbb{R}^p$ is the vector of regression coefficients. The affine subspace may be written as
\[
\spn\biggl(\begin{bmatrix} I_p \\ \beta^\tp  \end{bmatrix} \biggr) + \beta_{p+1} e_{p+1}  \in \Graff(p, n+1)
\]
where $I_p$ is the $p  \times p$ identity matrix. It best represents the data $(X,y)$ in the sense of linear regression. This description corresponds to how one usually pictures linear regression --- drawing an affine hyperplane through a collection of $n$ scattered data points $(x_i,y_i)^\tp \in \mathbb{R}^p \times \mathbb{R} =\mathbb{R}^{p+1}$, where $x_i$ is the $i$th row of $X$ and $y_i$ is the $i$th entry of $y$,  $i=1,\dots,n$.
\end{example}

\begin{example}[Errors-in-Variables Regression]
We follow the same notations as in the above example. Here we concatenate the explanatory variables and response variable and assign them equal weights. The best-fitting affine subspace of the data set $\{(x_i,y_i)^\tp \in \mathbb{R}^{p+1}: x_i \in \mathbb{R}^p, \; y_i \in \mathbb{R}, \; i = 1,\dots ,n\}$ in this case is given by
\[
\spn(w) + b \in \Graff(1, p+1)
\]
where  $w, b  \in  \mathbb{R}^{p+1}$ are the minimizer of the loss function $\sum_{i=1}^{n}  \lVert(I-ww^\tp )((x_i,y_i)^\tp-b) \rVert_F^2 $ subject to $ w^\tp b= 0$, and may be obtained by solving a total least squares problem. 
\end{example}

\begin{example}[Principal Component Analysis]\label{eg:pca}
Let $\overline{x}=\frac{1}{n}X^\tp \mathbbm{1}\in\mathbb{R}^{p}$ be the sample mean of a data matrix $X \in \mathbb{R}^{n \times p}$ so that $\overline{X} = X-\mathbbm{1}\overline{x}^\tp $ is mean-centered. For $k \le p$, the $k$th principal subspace is $\spn(Z_k)$, a $k$-dimensional linear subspace of $\mathbb{R}^p$ such that $Z_k \in \mathbb{R}^{p \times k}$ maximizes $\tr(Z_k^\tp 
\smash{\overline{X}}^\tp \overline{X} Z_k), $ subject to $Z^\tp _{k} Z_{k} = I_k.$ The affine subspace
\[
\spn(Z_k) + \overline{x} \in \Graff(k, p)
\]
captures the greatest $k$-dimensional variability in the data $X$. The $k$ largest principal components of $X$ are defined successively for $k=1,\dots,p$ as orthonormal basis of $\spn(Z_k)$.
\end{example}

\begin{example}[Support Vector Machine]\label{eg:svm}
Let $\{(x_i,y_i): x_i \in \mathbb{R}^p, \; y_i = \pm 1,\; i = 1,\dots ,n\}$ be a training set  for binary classification. The best separating hyperplane is given by $w^\tp x - \beta = 0$, where  $(w, \beta) \in \mathbb{R}^p \times \mathbb{R}$ can be found by minimizing $ \lVert w\rVert$ subject to $y_{i}(w^\tp  x_i - b) \ge 1$ for all $ i = 1,\dots ,n$. In other words, the best separating hyperplane is the affine subspace 
\[
\ker (w^\tp ) + \beta \mathbbm{1} \in \Graff(1,p).
\]
\end{example}
These four examples represent a sampling of the most rudimentary classical examples. It is straightforward to extend them to include more modern considerations. We may incorporate say, sparsity or robustness, by changing the objective function used; or have matrix variables in place of vector variables by considering affine subspaces within other vector spaces, e.g., with $\mathbb{S}^n$ or $\mathbb{R}^{m \times n}$ in place of $\mathbb{R}^n$.

These simple examples may be solved in the usual manners with techniques in numerical linear algebra: least squares for linear regression, singular value decomposition for errors-in-variables regression, eigenvalue decomposition for principal component analysis, linear programming for support vector machines.  Nevertheless, viewing them in their full generality as optimization problems on the affine Grassmannian allows us to treat them on equal footings and facilitates development of new multivariate statistics/machine learning techniques.
More importantly, we argue that the prevailing approaches may be suboptimal. For instance, in Example~\ref{eg:pca} one circumvents the problem of finding a best-fitting affine subspace with a two-step heuristic: First find the empirical mean of the data set $\overline{x}$ and then mean center to reduce the problem to one of finding a best-fitting linear subspace $\spn(Z)$. But there is no reason to expect $\spn(Z) + \overline{x}$ to be the best-fitting affine subspace. In \cite{WYL}, we developed various optimization algorithms --- steepest decent, conjugate gradient, Newton method --- that allow us to directly optimize real-valued functions defined on $\Graff(k,n)$.

We would like to highlight another reason we expect the  affine Grassmannian to be useful in data analytic problems. Over the past two decades, parameterizing a data set by geometric structures has become a popular alternative to probabilistic modeling, particularly when the intrinsic dimension of the data set is low or when it satisfies obvious geometric constraints.  In this case, statistical estimation  takes into account the intrinsic geometry of the data, and the deviation from the underlying geometric structures is used as a measure of accuracy of the statistical model. The two most common geometric structures employed are (a) a mixture of affine spaces \cite{HRS, LZ, MYDF} and (b) a manifold, which often  reduces to (a) when it is treated as a collection of tangent spaces \cite{TVF} --- in fact, the first manifold learning techniques \textsc{isomap} \cite{isomap}, \textsc{lle} \cite{LLE}, and Laplacian Eigenmap \cite{Laplace} are essentially different ways to approximate a manifold by a collection of its tangent spaces.  This provides another impetus for studying $\Graff(k,n)$, which parameterizes all affine spaces of a fixed dimension in an ambient space; $\Graff(k,\infty)$, which parameterizes all affine spaces of a fixed dimension; and $\Graff(\infty,\infty)$, which parameterizes all affine spaces of all dimensions.

\section*{Acknowledgment} 

The work in this article is generously supported by AFOSR FA9550-13-1-0133, DARPA D15AP00109, NSF IIS 1546413, DMS 1209136, DMS 1057064, National Key R\&D Program of China Grant no.~2018YFA0306702,  and NSFC Grant no.~11688101. In addition, LHL's work is  supported by a DARPA Director's Fellowship and the Eckhardt Faculty Fund; KY's work is supported by the Hundred Talents Program of the Chinese Academy of Sciences and the Recruitment Program of Global Experts of China.

\end{document}